\documentclass[a4paper]{amsart}
\usepackage{amsmath,amssymb,times, amscd, graphicx, xspace,mathrsfs,a4wide,hyperref,url,chngcntr}
\usepackage[title]{appendix}
\hypersetup{
  colorlinks   = true, %Colours links instead of ugly boxes
  urlcolor     = blue, %Colour for external hyperlinks
  linkcolor    = blue, %Colour of internal links
  citecolor   = blue %Colour of citations
}

\newcommand{\Z}{\ensuremath{\mathbb{Z}}\xspace}
\newcommand{\Q}{\ensuremath{\mathbb{Q}}\xspace}
\newcommand{\R}{\ensuremath{\mathbb{R}}\xspace}
\newcommand{\C}{\ensuremath{\mathbb{C}}\xspace}
\newcommand{\A}{\ensuremath{\mathbb{A}}\xspace}

\newcommand{\Qp}{\ensuremath{\mathbb{Q}_{p}}\xspace}
\newcommand{\Zp}{\ensuremath{\mathbb{Z}_{p}}\xspace}
\newcommand{\Fp}{\ensuremath{\mathbb{F}_{p}}\xspace}
\newcommand{\Sp}{\mathrm{Sp}\xspace}
\newcommand{\D}{\mathcal{D}}

\newcommand{\OO}{\ensuremath{\mathcal{O}}\xspace}

\newcommand{\Frob}{\mathrm{Frob}\xspace}

\newcommand{\comment}[1]{}

\DeclareMathOperator{\Gal}{Gal}
\DeclareMathOperator{\End}{End}
\DeclareMathOperator{\Hom}{Hom}

\DeclareMathOperator{\Spec}{Spec}

\DeclareMathOperator{\Spa}{Spa}

\DeclareMathOperator{\Ker}{Ker}

\DeclareMathOperator{\Max}{Max}
\DeclareMathOperator{\Ann}{Ann}
\newcommand{\Supp}{\ensuremath{\mathrm{Supp}}\xspace}

\newcommand{\T}{\ensuremath{\mathbb{T}}\xspace}

\newcommand{\GL}{\ensuremath{\mathrm{GL}}\xspace}

\newcommand{\cX}{\ensuremath{\mathscr{X}}\xspace}

\newcommand{\mbf}{\mathbf}
\newcommand{\mb}{\mathbb}
\newcommand{\mc}{\mathcal}
\newcommand{\ms}{\mathscr}
\newcommand{\mf}{\mathfrak}
\newcommand{\vp}{\varpi}
\newcommand{\ra}{\rightarrow}

\newcommand{\sub}{\subseteq}
\newcommand{\oo}{\mathcal{O}}
\newcommand{\ol}{\overline}
\newcommand{\bu}{\bullet}
\newcommand{\ka}{\kappa}

\newcommand{\wh}{\widehat}
\newcommand{\wt}{\widetilde}
\newcommand{\ok}{\mathcal{O}_{K}}

\newtheorem{theorem}{Theorem}[subsection]
\newtheorem{proposition}[theorem]{Proposition}
\newtheorem{corollary}[theorem]{Corollary}
\newtheorem{lemma}[theorem]{Lemma}

\newtheorem*{theor}{Theorem}

\theoremstyle{definition}
\newtheorem{definition}[theorem]{Definition}

\newtheorem{remark}[theorem]{Remark}

\mathchardef\mhyphen="2D
\title{Irreducible components of extended eigenvarieties and interpolating 
Langlands functoriality}
\author{Christian Johansson and James Newton}

\address{DPMMS, Centre for Mathematical Sciences, Wilberforce Road, Cambridge CB3 0WB, UK}
\email{hcj24@cam.ac.uk}
\address{Department of Mathematics, King's College London, London WC2R 2LS, UK}
\email{j.newton@kcl.ac.uk}
\input xy
\xyoption{all}
\usepackage{amscd}
\newcommand{\rk}{\mathrm{rk}\xspace}

\newcommand{\Qbar}{\overline{\mathbb{Q}}}

\begin{document}
\begin{abstract}
We study the basic geometry of a class of analytic adic spaces that arise in 
the study of the extended (or adic) eigenvarieties constructed by 
Andreatta--Iovita--Pilloni, Gulotta and the authors. We apply this to prove a general 
interpolation theorem for Langlands functoriality, which works for extended 
eigenvarieties and improves upon existing results in characteristic $0$. As an 
application, we show that the characteristic $p$ locus of the extended 
eigenvariety for ${\rm GL}_{2}/F$, where $F/\Q$ is a cyclic extension, contains 
non-ordinary components of dimension at least $[F:\Q]$ .
\end{abstract}
\maketitle

\counterwithin{equation}{subsection}
\section{Introduction}
\subsection{Previous work}
In a previous work \cite{jn}, we have described a general construction of 
`extended eigenvarieties': analytic adic spaces (of mixed characteristic), 
which contain the rigid analytic eigenvarieties constructed by Hansen 
\cite{han} as an open subspace. This follows work of Andreatta, Iovita and 
Pilloni \cite{aip}, who gave a (different) construction of an extended 
eigencurve, and there is also independent work of Gulotta \cite{gu} which constructs extended versions of Urban's eigenvarieties \cite{urb}. These newly constructed objects appear to be natural spaces in 
which to consider families of finite slope automorphic representations. 
Moreover, they provide a new perspective on the geometry of rigid 
analytic eigenvarieties. We refer to \cite{aip,jn} for more remarks and 
questions related to these extended eigenvarieties.

\medskip

One basic question about extended eigenvarieties is: do they contain rigid 
analytic eigenvarieties as a \emph{proper subset}? Results of Bergdall--Pollack 
\cite{bp} and Liu--Wan--Xiao \cite{lwx} show that the extended eigencurve (and 
an analogue 
constructed using definite quaternion algebras over $\Q$) does indeed contain 
the Coleman--Mazur eigencurve as a proper subset. More precisely, they show that the extended eigencurve contains infinitely many non-ordinary points in characteristic  $p$. One motivation for this 
article is to bootstrap this result, using a new result on $p$-adic interpolation of 
Langlands functoriality, to show that other extended eigenvarieties contain 
their rigid analytic counterpart as a proper subset, and give lower bounds on the dimension of the characteristic $p$ locus. 

\subsection{Contents of this article}
In this article, we begin by studying the types of analytic adic spaces that 
appear in the construction of extended eigenvarieties. We tentatively call 
them \emph{pseudorigid spaces}, since they generalise rigid spaces (over 
complete discretely valued fields) and have the same key features. Our main 
focus is to establish some basic results about the geometry of pseudorigid 
spaces, including the existence of a Zariski topology, normalizations, 
irreducible components and a well behaved dimension theory. We refer to 
\cite{con} for these notions in the setting of rigid analytic spaces, and our 
exposition is heavily influenced by this reference.

\medskip

We then define a very general and elementary version of an abstract `eigenvariety datum', which is more 
flexible than existing notions in the literature. Using it we prove  an 
interpolation theorem (Thm.~\ref{optimalint}) which has the following somewhat imprecise form: 
\begin{theor}
Suppose we have two extended eigenvarieties $\ms{X}, \ms{Y}$, a homomorphism 
between appropriate abstract Hecke algebras $\sigma: \T_{\ms{Y}} \rightarrow 
\T_{\ms{X}}$ and a collection of `classical points' $\ms{X}^{cl}$ whose 
associated systems of Hecke eigenvalues, when composed with $\sigma$, also 
appear in $\ms{Y}$. Then the associated map $\ms{X}^{cl} \rightarrow \ms{Y}$ 
interpolates into a canonical morphism from the Zariski closure of 
$\ms{X}^{cl}$ in $\ms{X}$ to $\ms{Y}$.  
\end{theor}

This result gives a substantial and essentially optimal generalisation of 
previous interpolation theorems which have
been used to study $p$-adic Langlands functoriality \cite[Prop.~3.5]{che2}, 
\cite[Prop.~7.2.8]{bc}, \cite[Thm.~5.1.6]{han} (see Remark 
\ref{rem:compareinterp} for more precise comments).

Somewhat surprisingly, both the 
statement and the proof of our theorem seem to be significantly simpler than existing 
results, and it seems to be very well-adapted to interpolating known cases of 
Langlands functoriality\footnote{More precisely (but still somewhat imprecisely), Langlands functoriality implies transfer of systems of Hecke eigenvalues, and our theorem and its predecessors allow for interpolation of the transfer of systems of Hecke eigenvalues. In general this information is coarser than the transfer of L-packets.}. In particular, our methods are rather different from those of Chenevier, Bellaiche--Chenevier and of Hansen, and ultimately rely on a simple trick that allows us to consider two eigenvarieties as Zariski closed subspaces inside a common, bigger, eigenvariety. After proving our interpolation theorem, we then briefly review our construction of extended 
eigenvarieties from \cite{jn}. We address, at the suggestion of a referee, a 
question not touched upon in \cite{jn} concerning the dependence of our 
construction on the choice of controlling operator, and we also generalise a 
result giving a lower bound for 
the dimensions of 
irreducible components of eigenvarieties \cite[Prop.~B.1]{han} to extended 
eigenvarieties (Prop.~\ref{prop:dimension}).

\medskip

Finally, as an example application, we apply the interpolation theorem to 
establish the existence of a base change map between the extended eigencurve and the 
extended eigenvarieties for $\GL_2$ over a cyclic extension $F/\Q$ 
(Thm.~\ref{padicBC}). Combining this with our result on the dimensions of 
irreducible components shows that the characteristic $p$ loci of
these extended eigenvarieties contain (non-ordinary) components of dimension at 
least $[F:\Q]$ (Cor.~\ref{cor:bigcharp}). In particular, they strictly contain 
their rigid analytic loci (where $p$ is invertible). We also outline how the 
same strategy would allow one to give generalisations of this result (Rem 
\ref{rem: extension}).

\medskip
To conclude this introduction, we remark that recent work of Louren\c{c}o \cite{lou} studies the geometry of pseudorigid spaces further, in particular extending results of Bartenwerfer and L\"utkebohmert on extending (bounded) functions on normal rigid spaces to normal pseudorigid spaces.

\subsection*{Acknowledgments}

The authors wish to thank David Hansen for useful conversations and for 
comments on 
earlier drafts of this paper, and Judith Ludwig for many helpful 
comments on several aspects of the paper. We also wish to thank the anonymous 
referee for useful comments, which prompted us to make simplifications in section \ref{sec interpolation} and to include the results of 
section \ref{independence}. C.J. was supported by NSF grant 
DMS-1128155 during the initial stage of this work.
\section{Irreducible components and normalization}

\subsection{The Zariski topology for locally Noetherian adic spaces}

For basics on adic spaces we refer to \cite{hu1, hu2, hu3}. Following now common terminology, we will say that an adic space $X$ is \emph{locally Noetherian} if it is locally of the form $\Spa(R,R^{+})$, where $R$ has a Noetherian ring of definition over which $R$ is finitely generated, or locally of the form $\Spa(S,S^{+})$, where $S$ is strongly Noetherian. To make things slightly easier, we will say that an affinoid ring $(A,A^{+})$ is Noetherian if A is strongly Noetherian, or admits a Noetherian ring of definition over which it is finitely generated. All affinoid rings in this paper will be assumed complete unless otherwise stated.

\medskip

Let us start with a general observation. If $(R,R^{+})$ is any affinoid ring, 
then there is a continuous map $\Spa(R,R^{+}) \ra \Spec(R)$ sending a valuation 
$v$ to its kernel (or support) $\Ker v$. This map is functorial in $(R,R^{+})$. 
One may define the \emph{Zariski topology} on $\Spa(R,R^{+})$ to be the 
topology whose open sets are the preimages of the open sets in $\Spec(R)$. It 
is not obvious to us how one would extend this construction to arbitrary adic 
spaces, since one cannot simply `glue' these topologies (if $\Spa(S,S^{+}) \sub 
\Spa(R,R^{+})$ is a rational subset, then the inclusion is continuous for the 
Zariski topology, but there is no reason that it should be an open embedding). 
Our goal in this section is to give a definition of the Zariski topology for 
arbitrary \emph{locally Noetherian} adic spaces which recovers the definition 
above in the affinoid case, and show that this gives a natural notion of 
irreducibility and irreducible components.

\medskip

Let $X$ be a locally Noetherian adic space. By \cite[Theorem 2.5]{hu2}, $X$ has a good theory of coherent $\oo_{X}$-modules. If $X=\Spa(R,R^{+})$ is affinoid with $(R,R^{+})$ Noetherian, then there is an equivalence of abelian categories between finitely generated $R$-modules and coherent $\oo_{X}$-modules, given by sending a coherent $\oo_{X}$-module $\mc{F}$ to its global sections $\mc{F}(X)$, and sending a finitely generated $R$-module $M$ to the sheaf $\wt{M}$ defined by
$$ \wt{M}(U)= M \otimes_{R} \oo_{X}(U)$$
whenever $U\sub X$ is a rational subset. We say that a coherent 
$\oo_{X}$-module $\mc{I}$ is a \emph{coherent $\oo_{X}$-ideal} (or 
\emph{coherent ideal} for short, if no confusion seems likely to arise) if it 
is a sub-$\oo_{X}$-module of $\oo_{X}$. By the construction in 
\cite[(1.4.1)]{hu3}, any coherent ideal $\mc{I}$ gives rise to a \emph{closed 
adic subset} \[V(\mc{I}) = \{x \in X: \mc{I}_x \neq \OO_{X,x}\}\] of $X$ (in 
particular, this is a closed subset of $X$). Locally, if $X=\Spa(R,R^{+})$ with 
$(R,R^{+})$ Noetherian and $\mc{I}=\wt{I}$, we have 
$V(\mc{I})=\Spa(R/I,(R/I)^{+})$, where $(R/I)^{+}$ is defined to be the 
integral closure of $R^{+}/(I \cap R^{+})$ in $R/I$. 

\medskip

Let us return to the case of general locally Noetherian $X$. We wish to check that the closed adic subsets of $X$ form the closed subsets of a topology. To do this, let us define some operations on coherent ideals. If $(\mc{I}_{j})_{j\in J}$ is a collection of coherent ideals, then we define $\sum_{j\in J}\mc{I}_{j}$ to be the sheaf associated with the presheaf
$$ U \mapsto \sum_{j\in J}\mc{I}_{j}(U)\sub \oo_{X}(U). $$
It is a subsheaf of $\oo_{X}$ by construction. If $\mc{I}_{1}$ and $\mc{I}_{2}$ are two coherent ideals, we define their intersection $\mc{I}_{1}\cap \mc{I}_{2}$ to be the sheaf $U \mapsto \mc{I}_{1}(U)\cap \mc{I}_{2}(U)$. Note that these constructions commute with restriction to open subsets. We record the following elementary lemma.

\begin{lemma}\label{zar1}
Let $(R,R^{+})$ be Noetherian. If $(I_{j})_{j\in J}$ is a collection of ideals of $R$ and $I=\sum_{j\in J}I_{j}$, then $\wt{I}=\sum_{j\in J}\wt{I}_{j}$. If $I_{1}$ and $I_{2}$ are two ideals of $R$, then $\wt{I_{1}\cap I_{2}} = \wt{I}_{1}\cap \wt{I}_{2}$.
\end{lemma}

\begin{proof}
The second statement is straightforward to deduce from the definitions, so we 
content ourselves with proving the first statement. By the definitions, 
$\wt{I}$ is the sheafification of a presheaf which on rational subsets $U$ is 
given by
$$ U \mapsto \sum_{j\in J} (I_{j} \otimes_{R} \oo_{X}(U)). $$
By flatness of $R \ra \oo_{X}(U)$, we have $I_{j} \otimes_{R} \oo_{X}(U) = I_{j}\oo_{X}(U)$ for all $j\in J$, so it suffices to prove that $\sum_{j\in J} I_{j}\oo_{X}(U) = I\oo_{X}(U)$ (since the latter is equal to $I\otimes_{R} \oo_{X}(U)$, again by flatness). By definition $I$ is the image of the the natural map $\bigoplus_{j\in J}I_{j} \ra R$, and $\sum_{j\in J} I_{j}\oo_{X}(U)$ is the image of the same map after applying $-\otimes_{R} \oo_{X}(U)$. The statement now follows from flatness of $R \ra \oo_{X}(U)$. 
\end{proof}

Let us now prove that when $X=\Spa(R,R^{+})$ with $(R,R^{+})$ Noetherian, the Zariski closed subsets of $\Spa(R,R^{+})$ (according to our general definition) are exactly the closed adic subsets.

\begin{proposition}\label{zar2}
Let $(R,R^{+})$ be Noetherian and put $X=\Spa(R,R^{+})$. Let $I$ be an ideal of $R$. Then $V(\wt{I})$ is the preimage of the closed subset of $\Spec(R)$ corresponding to $I$. 
\end{proposition}

\begin{proof}
Since all ideals of $R$ are closed and $R/I$ carries the quotient topology coming from $R$, this is essentially trivial; any $v\in X$ belongs to $\Spa(R/I,(R/I)^{+})$ if and only if $I\sub \Ker v$. 
\end{proof}

\begin{corollary}\label{zar3}
Let $X$ be a locally Noetherian adic space.
\begin{enumerate}
\item $V(0)=X$, and $V(\oo_{X})=\emptyset$.

\item If $(\mc{I}_{j})_{j\in J}$ is a collection of coherent ideals, then $\mc{I}=\sum_{j\in J}\mc{I}_{j}$ is a coherent ideal, and $V(\mc{I}) = \bigcap_{j\in J} V(\mc{I}_{j})$.

\item If $\mc{I}_{1}$ and $\mc{I}_{2}$ are two coherent ideals, then $\mc{I}_{1}\cap \mc{I}_{2}$ is a coherent ideal and $V(\mc{I}_{1} \cap \mc{I}_{2}) = V(\mc{I}_{1}) \cup V(\mc{I}_{2})$.
\end{enumerate}
\end{corollary}

\begin{proof}
It is enough to prove these statements locally, so we may assume that $X=\Spa(R,R^{+})$ with $(R,R^{+})$ Noetherian. The Corollary then follows directly from Lemma \ref{zar1} and Proposition \ref{zar2}.
\end{proof}

It now makes sense to make the following definition.

\begin{definition}
Let $X$ be a locally Noetherian adic space. We define the \emph{Zariski topology} of $X$ to be the topology on $X$ whose closed sets are the closed adic subsets of $X$.
\end{definition}

In light of this definition we will refer to closed adic subsets of $X$ as 
\emph{Zariski closed subsets of $X$}. Next, we will discuss more general 
coherent sheaves. Let $X$ be a locally Noetherian adic space and let $\mc{F}$ 
be a coherent $\oo_{X}$-module. We define its support $\Supp(\mc{F}$) by $\{ x\in 
X \mid \mc{F}_{x}\neq 0\}$. Note that $V(\mc{I}) = \Supp(\OO_X/\mc{I})$. 

\medskip

We now show that $\Supp(\mc{F})$ is closed (in 
the usual topology). It is enough to verify this locally so assume 
$X=\Spa(R,R^{+})$ with $(R,R^{+})$ Noetherian and $\mc{F}=\wt{M}$. Suppose $x 
\notin \Supp(\mc{F})$. We have $\wt{M}_{x}=\varinjlim_{x\in 
U}M\otimes_{R}\oo_{X}(U)=0$ where $U$ is rational. If $m_{1},..,m_{r}$ are 
generators of $M$, we see that there is a $U$ such that the $m_{i}$ vanish in 
$M\otimes_{R}\oo_{X}(U)$. It follows that $\mc{F}(U)=M\otimes_{R}\oo_{X}(U)=0$, 
and we deduce that $\Supp(\mc{F})$ is closed. Our next goal is to show that it 
is 
Zariski closed. We define the \emph{annihilator} $\Ann(\mc{F})$ of $\mc{F}$ by
$$ (\Ann(\mc{F}))(U) = \{ f\in \oo_{X}(U) \mid f\mc{F}_{x}=0\,\, \forall x\in U\}. $$
This is an $\oo_{X}$-subsheaf of $\oo_{X}$, and (since $\mc{F}$ is coherent) we 
have that 
$\Ann(\mc{F})_{x}=\Ann(\mc{F}_{x})$, where the right hand side is the 
annihilator of the $\oo_{X,x}$-module $\mc{F}_{x}$.

\begin{lemma}\label{coh1}
Let $(R,R^{+})$ be Noetherian and let $M$ be a finitely generated $R$-module. Then $\Ann(\wt{M})=\wt{\Ann(M)}$.
\end{lemma}

\begin{proof}
Recall that the formation of annihilators of finitely generated modules commute 
with flat base change \cite[Tag 07T8]{stacks-project}. Let $U$ be a rational 
subset of $X=\Spa(R,R^{+})$. First note that 
$\Ann(M)\oo_{X}(U)=\Ann(M\otimes_{R}\oo_{X}(U))\sub \Ann(\wt{M})(U)$, where we 
have used the above fact for the first equality. We need to prove the opposite 
inclusion. Let $f\in \Ann(\wt{M})(U)$ and let $x\in U$. $f$ kills 
$\wt{M}_{x}=\varinjlim_{x\in V\sub U} M\otimes_{R}\oo_{X}(V)$, and hence (by 
finite generation) has to kill $M\otimes_{R}\oo_{X}(U_{x})$ for some rational 
subset $U_{x}\sub U$ containing $x$. It follows that we may find an open cover 
$U=\bigcup_{x\in U}U_{x}$ of rational subsets such that $f\in 
\Ann(M\otimes_{R}\oo_{X}(U_{x}))=\Ann(M)\oo_{X}(U_{x})$. This implies that 
$f\in \Ann(M)\oo_{X}(U)$, as desired.
\end{proof}

\begin{corollary}\label{coh2}
Let $X$ be a locally Noetherian adic space and $\mc{F}$ a coherent 
$\oo_{X}$-module. Then $\Ann(\mc{F})$ is coherent $\oo_{X}$-ideal and 
$\Supp(\mc{F})=V(\Ann(\mc{F}))$, so $\Supp(\mc{F})$ is Zariski closed. If $f : 
X^{\prime} \ra X$ is a morphism of locally Noetherian adic spaces and $Z\sub X$ 
is Zariski closed, then $f^{-1}(Z)\sub X^{\prime}$ is Zariski closed. If $f$ is 
finite and $Z^{\prime}\sub X^{\prime}$ is Zariski closed, then 
$f(Z^{\prime})\sub X$ is Zariski closed.
\end{corollary}

\begin{proof}
$\Ann(\mc{F})$ is a coherent ideal by Lemma \ref{coh1}. To see that 
$\Supp(\mc{F})=V(\Ann(\mc{F}))$, note that, for any $x\in X$, 
$\Ann(\mc{F})_{x}=\Ann(\mc{F}_{x})$ and that $\Ann(\mc{F}_{x})=\oo_{X,x}$ if 
and only if $\mc{F}_{x}=0$. For the second statement, choose a coherent 
$\oo_{X}$-ideal $\mc{I}$ such that $Z=V(\mc{I})$. Then 
$f^{-1}(Z)=\Supp(f^{\ast}(\OO_X/\mc{I}))$ is Zariski closed by the first 
statement of the corollary.

\medskip
It remains to prove the final statement. Replacing $X^{\prime}$ by $Z^{\prime}$ (choosing some coherent $\oo_{X^{\prime}}$-ideal $\mc{J}$ with $Z^{\prime}=V(\mc{J})$) we may assume that $X^{\prime}=Z^{\prime}$. We wish to prove that $f(X^{\prime})=\Supp(f_{\ast}\oo_{X^{\prime}})$. Since $f$ is finite, $f(X^{\prime})$ is closed, so we must have $\Supp(f_{\ast}\oo_{X^{\prime}}) \sub f(X^{\prime})$. If the inclusion is proper, we can find an open affinoid set $U\sub X$ which intersects $f(X^{\prime})$ but is disjoint from $\Supp(f_{\ast}\oo_{X^{\prime}})$, since $\Supp(f_{\ast}\oo_{X^{\prime}})$ is closed. But then we must have $(f_{\ast}\oo_{X^{\prime}})|_{U}=0$, but also $(f_{\ast}\oo_{X^{\prime}})(U)=\oo_{X^{\prime}}(f^{-1}(U))\neq 0$ since $f^{-1}(U)\neq \emptyset$, a contradiction. This finishes the proof.
\end{proof}

In particular, if $f : X \ra Y$ is a morphism of locally Noetherian adic spaces, then it is continuous with respect to the Zariski topologies on $X$ and $Y$. Let us record a few basic general topology facts about the Zariski topology. Recall that a topological space $Z$ is said to be \emph{Noetherian} if it satisfies the descending chain condition for closed subsets. We say that a topological space is \emph{locally Noetherian} if it has an open cover by Noetherian spaces. Note that if a topological space is quasicompact and locally Noetherian, then it is Noetherian.

\begin{lemma}\label{irr1}
Let $X$ be a locally Noetherian adic space.
\begin{enumerate}
\item If $X=\Spa(R,R^{+})$ is affinoid with $(R,R^{+})$ Noetherian, then $X$ is Noetherian for the Zariski topology.

\item $X$ is locally Noetherian for the Zariski topology, and hence if $X$ is quasicompact then $X$ is Noetherian for the Zariski topology.
\end{enumerate}
\end{lemma}

\begin{proof}
For (1), note that if $(R,R^{+})$ is Noetherian then $R$ is a Noetherian ring, so the statement follows from Proposition \ref{zar2} and the fact that $\Spec(R)$ is Noetherian. For (2), note that if $U=\Spa(S,S^{+})\sub X$ is an open affinoid with $(S,S^{+})$ Noetherian, then the subspace topology on $U$ coming from the Zariski topology on $X$ is coarser than the Zariski topology on $U$. It follows from (1) that $U$ is Noetherian with respect to the subspace topology, which proves that $X$ is locally Noetherian.
\end{proof}

Let us start to discuss irreducible components, by making the following standard definition:

\begin{definition}
Let $X$ be a locally Noetherian adic space. A Zariski closed subset $Z$ is said be \emph{irreducible} if, for any Zariski closed subsets $Z_{1},Z_{2}$ of $X$, $Z\sub Z_{1}\cup Z_{2}$ implies $Z\sub Z_{1}$ or $Z\sub Z_{2}$. $Z$ is said to be an \emph{irreducible component of $X$} if it is irreducible and not properly contained in any other irreducible Zariski closed set. 
\end{definition}

There are a number of things we would expect from a satisfactory theory of 
irreducible components. In particular, we would like to be able to write $X$ as 
the union of its irreducible components (and that the union of any proper 
subset of irreducible components is not the whole of $X$), and we would like to 
be able to give any Zariski closed subset a \emph{canonical} structure of a 
reduced locally Noetherian adic space. For now, we will not say so much about 
the first question, except to note that if $X$ is quasicompact, then the 
Zariski topology is Noetherian by above and we have a satisfactory notion of 
irreducible components. Our intended applications, however, are to 
eigenvarieties, which are not quasicompact.

\subsection{Pseudorigid spaces}

Let $K$ be a complete discretely valued field with ring of integers $\oo_{K}$, a uniformizer $\pi_{K}$ and residue field $k$. We start by giving a name to the family of Tate rings over $\oo_{K}$ that we will be working with in this paper. Whenever we drop $R^{+}$ from the notation, it is assumed that $R^{+}=R^{\circ}$. Accordingly, we will often conflate an f-adic ring $R$ with the affinoid ring $(R,R^{\circ})$.

\begin{definition}\label{ps1}
Let $R$ be a Tate $\oo_{K}$-algebra. We say that $R$ is \emph{a Tate ring 
formally of finite type over $\oo_{K}$}\footnote{In \cite{lou}, these are 
called pseudoaffinoid $\oo_K$-algebras.} if $R$ has a ring of definition 
$R_{0}$ which is formally of finite type over $\oo_{K}$ in the usual sense 
(i.e. has a radical ideal of definition $I$ such that $R_{0}/I$ is a finitely 
generated $k$-algebra).
\end{definition}
\begin{remark}
	The topology on a ring of definition $R_0$ for a Tate 
	$\oo_{K}$-algebra $R$ is the $(\vp)$-adic topology for some topologically 
	nilpotent unit $\vp \in R$. We obtain a map $\oo_K[[X]]\rightarrow R_0$ 
	defined by $X\mapsto\vp$ and 
	$R_0$ is formally of finite type over $\oo_K$ if and 
	only if it is topologically of finite type over $\oo_K[[X]]$, i.e.~if and 
	only if it is (topologically) isomorphic to a quotient of 
	$\oo_K[[X]]\langle Y_1,\ldots,Y_n\rangle$ for some $n$.
\end{remark}

We note that any ideal $I$ of a Tate ring $R$ formally of finite type over $\ok$ is closed, and that if $S$ is an f-adic ring which is topologically of finite type over $R$ then $S$ is a Tate ring formally of finite type over $\ok$ as well. In particular, all quotients and all rational localizations of $R$ are Tate rings formally of finite type over $\ok$. Let us recall a few results about Tate rings formally of finite type over $\ok$ from \cite[Appendix A]{jn}.

\begin{proposition}\label{ps2}
Let $R$ and $S$ be Tate rings formally of finite type over $\ok$, and let $f : 
R\ra S$ be continuous homomorphism which induces an open immersion 
$\Spa(S,S^{+}) \ra \Spa(R)$ for some ring $S^{+}$ of integral elements in $S$. 
Then the following hold:
\begin{enumerate}
\item $R$ is Jacobson and excellent.

\item $S^{+}=S^{\circ}$.

\item If $\mf{m}$ is a maximal ideal of $S$, then $f^{-1}(\mf{m})$ is a maximal ideal of $R$, and the natural map $R_{f^{-1}(\mf{m})} \ra S_{\mf{m}}$ induces an isomorphism on completions.

\item\label{ps2finite} If $R \ra T$ is finite and $(R,R^{\circ}) \ra (T,T^{+})$ 
is the corresponding finite morphism of affinoid rings, then $T^{+}=T^{\circ}$.

\item If $R$ is reduced, then $S$ is reduced.
\end{enumerate}
\end{proposition}

\begin{proof}
This consists of collecting well known results together with results from \cite{jn} (which in turn are deduced mostly from well known results and results from \cite{ab}). We start with part (1). $R$ is Jacobson by \cite[Lemma A.1]{jn}. If $R_{0}$ is a ring of definition of $R$ which is formally of finite type over $\ok$, then $R_{0}$ is excellent by \cite[Proposition 7]{v1} and \cite[Theorem 9]{v2}, and hence $R$ is excellent since it is finitely generated over $R_{0}$. Part (2) follows from \cite[Theorem A.7]{jn} in the reduced case, but the argument to show that $S^{+}=S^{\circ}$ works without the reducedness assumption. Part (3) follows from \cite[Proposition A.15(2)]{jn}. Part (4) follows from \cite[Lemma A.3]{jn}. Part (5) follows from parts (1) and (2) together with standard properties reducedness and excellence, cf. the first paragraph of \cite[\S 1.2]{con}, or alternatively from \cite[Theorem A.7]{jn}.
\end{proof}

Next, we make a global definition.

\begin{definition}\label{ps3}
Let $X$ be an adic space over $\oo_{K}$. We say that $X$ is a \emph{pseudorigid space over $\oo_{K}$} if it is locally of the form $\Spa(R)$, for $R$ a Tate ring formally of finite type over $\oo_{K}$.
\end{definition}

We remark that any rigid space over $K$ is a pseudorigid space over $\oo_{K}$, where by rigid space over $K$ we mean an adic space $X$ over $K$ which is locally of the form $\Spa(R)$, for $R$ topologically of finite type over $K$. We will say that $X$ is an \emph{affinoid pseudorigid space over $\ok$} if it is equal to $\Spa(R)$ for some $R$ which is a Tate ring formally of finite type over $\ok$. Our next goal is to single out the points on a pseudorigid space that locally come from maximal ideals.

\begin{lemma}\label{ps4}
Let $R$ be a Tate ring formally of finite type over $\ok$, and let $\mf{m}\sub R$ be a maximal ideal. Then there exists a unique $v\in \Spa(R)$ with $\Ker v=\mf{m}$. This gives a canonical embedding of the spectrum $\Max(R)$ of maximal ideals into $\Spa(R)$, compatible with the morphism $\Spa(R) \ra \Spec(R)$. 
\end{lemma} 

\begin{proof}
The locus of $v\in \Spa(R)$ with $\Ker v=\mf{m}$ is equal to 
$\Spa(R/\mf{m})$ (here we use Prop.~\ref{ps2}(4)), so we have to prove 
that this is a singleton. Let $R_{0}$ be a ring of definition formally of 
finite type over $\ok$, and let $\vp\in R$ be a topologically nilpotent unit, 
which we assume lies in $R_{0}$. Let $\mf{p}=R_{0}\cap \mf{m}$. The quotient 
ring $R/\mf{m}$ carries the topology where $R_{0}/\mf{p}$ is open and carries 
the $\ol{\vp}$-adic topology, where $\ol{\vp}$ is the reduction of $\vp$. Since 
$\mf{p}$ is closed point in $\Spec(R_{0})\setminus \{\vp=0\}=\Spec(R)$, 
$R_{0}/\mf{p}$ is a $1$-valuative order. Here we refer to \cite[Definition 
1.11.1]{ab} or \cite[Definition A.11]{jn} for the definition of a $1$-valuative 
order (at least in the Noetherian setting), and \cite[Proposition 1.11.8]{ab} 
for the implication (which is an equivalence). As a result, 
$R/\mf{m}=(R_{0}/\mf{p})[1/\ol{\vp}]$ is a complete discretely valued field, 
with valuation ring the integral closure of $R_{0}/\mf{p}$, and $R_{0}/\mf{p}$ 
is open in the valuation topology. It follows that the valuation topology is 
equal to the quotient topology coming from $R$, and hence that $\Spa R/\mf{m}$ 
is a singleton, as desired. This proves the first statement, and the remaining 
statements are immediate.
\end{proof}

\begin{proposition}\label{ps5}
Let $R$ and $S$ be Tate rings formally of finite type over $\ok$, and let $f : R \ra S$ be a continuous $\ok$-homomorphism. Then $f$ is topologically of finite type, and preimages of maximal ideals are maximal ideals. Hence the induced map $\phi : \Spa(S) \ra \Spa(R)$ maps $\Max(S)$ into $\Max(R)$. Moreover, $\Max(R)$ is dense in $\Spa(R)$.
\end{proposition}

\begin{proof}
The first part is \cite[Corollary A.14]{jn} (the statement is only for $K=\Qp$, but the proof works the same). The second part follows immediately. For the third part, note that if $U\sub \Spa(R)$ is a rational subset, then the first part implies that $\Max(\oo_{X}(U))=U\cap \Max(R)$. It follows that $U\cap \Max(R)$ is non-empty if $U$ is, and hence that $\Max(R)$ is dense.
\end{proof}

Armed with this we may generalize the notion of `classical points' on a rigid space to pseudorigid spaces. 

\begin{definition}\label{ps6}
Let $X$ be a pseudorigid space. We say that a point $x\in X$ is \emph{maximal} if there exists an open affinoid pseudorigid neighbourhood $U=\Spa(R)$ of $x$ such that $x\in \Max(R)\sub U$. We denote the set of maximal points by $\Max(X)$.
\end{definition}

If $x\in \Max(X)$ and $V=\Spa(S)\sub X$ is an open affinoid pseudorigid neighbourhood, then by Proposition \ref{ps5} $x\in \Max(S)$. Moreover, it also follows that $\Max(X)$ is dense, that $\Max(X)=\Max(R)$ if $X=\Spa(R)$ is affinoid pseudorigid, and that any morphism $f : X \ra Y$ of pseudorigid spaces maps $\Max(X)$ into $\Max(Y)$.

\medskip

Let us now discuss the Zariski topology for pseudorigid spaces. To do this, we 
need the radical of a coherent ideal, following \cite[(9.5.1)]{bgr}. Let $X$ be 
a pseudorigid space and let $\mc{I}$ be a coherent $\oo_{X}$-ideal. We define 
the radical $\sqrt{\mc{I}}$ of $\mc{I}$ to be the sheaf on $X$ associated with 
the presheaf
$$ U \mapsto \sqrt{\mc{I}(U)}, $$
where, if $A$ is any ring and $J\sub A$ is an ideal, $\sqrt{J}$ denotes the radical of $J$. This construction commutes with restriction to open subsets, and defines an $\oo_{X}$-ideal, which we would like to be coherent. One checks, using that the formation of the (usual) radical commutes with filtered direct limits, that \cite[\S9.5.1 Proposition 1]{bgr} goes through in our setting, with the same proof. In particular, \emph{if $\sqrt{\mc{I}}$ is coherent}, then $V(\mc{I})=V(\sqrt{\mc{I}})$ and $\sqrt{\mc{I}}|_{U}$ is the sheaf attached to the ideal $\sqrt{\mc{I}(U)}$ for any affinoid open $U=\Spa(S,S^{+})$ with $(S,S^{+})$ Noetherian. So, to prove that $\sqrt{\mc{I}}$ is coherent, we may reduce to the affinoid case.

\begin{proposition}\label{irr2}
Let $X=\Spa(R)$ be an affinoid pseudorigid space. Let $I\sub R$ be an ideal and put $\mc{I}=\wt{I}$. Then $\sqrt{\mc{I}}$ is a coherent ideal, associated with $\sqrt{I}$.
\end{proposition}

\begin{proof}
The argument in the proof of \cite[\S9.5.1 Proposition 2]{bgr} goes through, if we can verify that, for any rational subset $U\sub X$, $\sqrt{I}\oo_{X}(U)$ is a radical ideal in $\oo_{X}(U)$. But $\oo_{X}(U)/\sqrt{I}\oo_{X}(U)$ is a rational localization of the reduced Tate ring $R/\sqrt{I}$, so it is reduced by Proposition \ref{ps2}(5), which is what we wanted to prove.
\end{proof}

We also have the following familiar property of the radical.

\begin{proposition}\label{ps7}
Let $X$ be a pseudorigid space over $\ok$ and let $\mc{I}, \mc{J}$ be two coherent $\oo_{X}$-ideals with $V(\mc{I})=V(\mc{J})$. Then $\sqrt{\mc{I}}=\sqrt{\mc{J}}$.
\end{proposition}

\begin{proof}
It suffices to prove this locally, so we may reduce to the case $X=\Spa(R)$ and $\mc{I}=\wt{I}$, $\mc{J}=\wt{J}$ with $I,J$ ideals of $R$. If $\Spa(R/I)=V(\mc{I})=V(\mc{J})=\Spa(R/J)$ as subsets of $\Spa(R)$, it follows that $\Max(R/I)=\Max(R/J)$ as subsets of $\Spec(R)$. Since $R/I$ and $R/J$ are Jacobson (by Proposition \ref{ps2}(1)), $\Max(R/I)$ is dense in $\Spec(R/I)$ and similarly for $R/J$, so we we deduce that $\Spec(R/I)=\Spec(R/J)$ and hence that $\sqrt{I}=\sqrt{J}$, which finishes the proof. 
\end{proof}

We formulate the upshot of the previous two propositions in the following 
Corollary.

\begin{corollary}
Let $X$ be a pseudorigid space over $\ok$. Then any Zariski closed subset $Z=V(\mc{I})$ of $X$ has a canonical structure of a reduced locally Noetherian adic space, with structure sheaf $\oo_{X}/\sqrt{\mc{I}}$. We call this the \emph{reduced structure} on $Z$.
\end{corollary}

\begin{proof}
$\sqrt{\mc{I}}$ is coherent and $Z=V(\sqrt{\mc{I}})$ by Proposition \ref{irr2} and the discussion preceding it, and $\sqrt{\mc{I}}$ only depends on $Z$ (and not on the choice of $\mc{I}$) by Proposition \ref{ps7}.
\end{proof}

\subsection{Normalizations of pseudorigid spaces}
The goal of this subsection is to construct a theory of normalizations for pseudorigid spaces over $\oo_{K}$.

\begin{definition}\label{no1}
Let $X$ be a pseudorigid space over $\oo_{K}$. We say that $X$ is \emph{normal} if $X$ is locally of the form $\Spa(R)$, where $R$ is a normal Tate ring formally of finite type over $\ok$.
\end{definition}

This definition is well behaved, in the following sense.

\begin{lemma}\label{no2}
Let $X=\Spa(R)$ be an affinoid pseudorigid space over $\ok$.
\begin{enumerate}
\item If $R$ is normal and $U\sub X$ is a rational subset, then $\oo_{X}(U)$ is normal.

\item If $X$ is normal, then $R$ is normal.
\end{enumerate}
\end{lemma}

\begin{proof}
This follows from Proposition \ref{ps2} by standard arguments. We start with 
(1). Put $S=\oo_{X}(U)$ and let $\mf{n}$ be any maximal ideal of $S$. It 
suffices to prove that $S_{\mf{n}}$ is normal. Let $\mf{m}$ be the preimage of 
$\mf{n}$ in $R$. This is a maximal ideal and the map $R_{\mf{m}}\ra S_{\mf{n}}$ 
induces an isomorphism on completions. Since $R$ is normal and excellent, so is 
$R_{\mf{m}}$ and hence its completion. By excellence of $S_{\mf{n}}$, 
$S_{\mf{n}}$ is therefore normal as well, which finishes the proof. The same 
argument, but reversed, proves (2).
\end{proof}

We will need the following lemma, which is a partial generalization of \cite[Lemma 2.1.4]{con}.

\begin{lemma}\label{no4}
Let $X$ be a normal connected pseudorigid space over $\ok$ and let $Z\sub X$ be a Zariski closed subset. If $Z$ contains a nonempty open subset of $X$, then $Z=X$.
\end{lemma}

\begin{proof}
Let us first assume that $X=\Spa(R)$ is affinoid pseudorigid, in which case $R$ is normal by Lemma \ref{no2}. Since $X$ is connected, $R$ is in fact a normal domain. Let $Z=V(\wt{I})$ for some ideal $I$ of $R$ and assume that it contains an open subset $U\sub X$, which we may assume to be a connected rational subset. Set $S=\oo_{X}(U)$; $S$ is a normal domain. Pick $f\in I$, we want to show that $f=0$. Since $f$ vanishes on $Z$, it must map to $0$ in $S$, so it suffices to prove that the map $R \ra S$ is injective. Pick any maximal ideal $\mf{n}$ of $S$ and let $\mf{m}$ be its preimage in $R$. Composing $R \ra S$ with the natural map $S \ra \wh{S_{\mf{n}}}$, it suffices to prove that $R \ra \wh{S_{\mf{n}}}$ is injective. This morphism factors as
$$ R \ra R_{\mf{m}} \ra \wh{R_{\mf{m}}} \ra \wh{S_{\mf{n}}}$$
so it suffices to prove that these three maps are injective. The first is injective since $R$ is a domain, the second is injective by Krull's intersection theorem, and the third is an isomorphism by Proposition \ref{ps2}(3). This finishes the proof in the affinoid case.

\medskip

We now do the general case. Let $\mc{I}$ be a coherent ideal such that $Z=V(\mc{I})$; we wish to show that $\mc{I}=0$. Define $\Sigma$ to be the set of open nonempty affinoid pseudorigid subspaces $V\sub X$ such that $\mc{I}|_{V}=0$, and define $\Delta$ to be the set of open nonempty affinoid pseudorigid subspaces $W\sub X$ such that $\mc{I}|_{W}\neq 0$. Clearly $\Sigma \cap \Delta = \emptyset$, and we claim that if $V\in \Sigma$ and $W\in \Delta$ then $V\cap W = \emptyset$. Assume not, and let $G\sub V \cap W$ be nonempty open affinoid pseudorigid. Since $G\sub V$ we must have $\mc{I}|_{G}=0$. But this means that $Z\cap W$, which is Zariski closed in $W$ but not equal to $W$, contains a nonempty open subset $G\sub W$, which contradicts the Lemma in the affinoid case. It follows that $V \cap W =\emptyset$. Now set $U_{\Sigma}=\bigcup_{V\in \Sigma}V$ and $U_{\Delta}=\bigcup_{W\in \Delta}W$. These are disjoint open subsets and $X=U_{\Sigma}\cup U_{\Delta}$. By assumption we know that $U_{\Sigma}\neq \emptyset$, so by connectedness of $X$ we must have $\Delta=\emptyset$, and hence $\mc{I}=0$ and $Z=X$.
\end{proof}

Note that Lemma \ref{no4} implies that a normal connected quasicompact pseudorigid space $X$ is irreducible. Indeed, normality implies that it is reduced, and if we write $X$ as union $X=Z_{1}\cup Z_{2}$ of two Zariski closed subsets with $Z_{1}\neq X$, then $Z_{2}$ contains the nonempty open subset $X\setminus Z_{1}$ and therefore has to equal $X$.

\medskip

Let $R$ be a Tate ring formally of finite type over $\ok$. We denote by $\wt{R}$ the \emph{normalization} of $R$, i.e. the integral closure of $R$ in its total ring of fractions. Since $R$ is excellent, $\wt{R}$ is finite over $R$ and hence a Tate ring formally of finite type over $\ok$. If $X=\Spa(R)$, we put $\wt{X}=\Spa(\wt{R})$ and call it, together with its canonical map $p:\wt{X} \ra X$, the \emph{normalization of $X$}. The morphism $p : \wt{X} \ra X$ is finite and surjective (to see that it is surjective, use for example that $p$ is closed since it is finite, and maps $\Max(\wt{X})$ onto $\Max(X)$). Next we show this construction glues, for which we need to verify that it commutes with rational localization.

\begin{lemma}\label{no6}
Let $R$ be a Tate ring formally of finite type and $S$ a rational localization of $R$. Then there is a natural isomorphism $\wt{R}\otimes_{R}S \cong \wt{S}$, which is unique over $S$.
\end{lemma} 

\begin{proof}
Let $I$ be the nilradical of $R$ (which is also the kernel of $R\ra \wt{R}$). 
By the theory of the radical of coherent $\oo_{\Spa(R)}$-ideals developed 
above, $IS$ is the nilradical of $S$. One then checks that $\wt{R}\otimes_{R}S 
\cong \wt{R} \otimes_{R/I} S/IS$ and that $\wt{S/IS}=\wt{S}$, so we may reduce 
to the case when $R$, and hence $S$, is reduced.

\medskip

It then suffices to show that $S \ra \wt{R}\otimes_{R}S$ is a normalization, 
since a normalization of $S$ is unique up to unique isomorphism over $S$. To do 
this we will verify that conditions of \cite[Theorem 1.2.2]{con} are satisfied. 
$R$ is Japanese since it is excellent, and $S$ is flat over $R$ since it is a 
rational localization of $R$. Note that $\wt{R}\otimes_{R}S$ is normal by Lemma 
\ref{no2} since it is a rational localization of $\wt{R}$. The last thing to 
verify is that if $\mf{p}$ is a minimal prime of $R$, then $S/\mf{p}S$ is 
reduced, which follows from Proposition \ref{ps2}(5) since it is a rational 
localization of $R/\mf{p}$.
\end{proof}

We may then globalize the construction.

\begin{definition}
Let $X$ be a pseudorigid space over $\ok$. Then we may construct the \emph{normalization $\wt{X} \ra X$} by gluing together the normalizations of the open affinoid pseudorigid subspaces of $X$. The canonical map $p : \wt{X} \ra X$ is finite and surjective. 
\end{definition}

Using normalizations, we can deduce the following strengthening of Lemma \ref{no4}.

\begin{lemma}\label{no7}
Let $X$ be an irreducible pseudorigid space over $\ok$. Then $\wt{X}$ is connected. Let $Z\sub X$ be a Zariski closed subset. If $Z$ contains a nonempty open subset of $X$, then $Z=X$.
\end{lemma}

\begin{proof}
	Let $p: \wt{X} \ra 
	X$ be the normalization. 
	Assume that $\wt{X} = U \coprod V$ with $U, V$ open. We claim that  
	$p^{-1}(p(U))-U$ is nowhere dense in $\wt{X}$
	(i.e.~contains no nonempty open subset). For now we assume the claim. 
	Observe that $U$ and $V$ are 
	Zariski open and closed. Since $X = p(U)\cup p(V)$ and $X$ is irreducible 
	we have $X = p(U)$ or $X = p(V)$. We may as well assume $X = p(U)$. Since 
	$p^{-1}(p(U))-U = \wt{X} - U$ is nowhere dense, we deduce that $V$ is 
	empty. We now check the claim. It suffices to do this locally on $X$, so we 
	assume that $X=\Spa(R)$ is an affinoid pseudorigid space (not necessarily 
	irreducible). $U$ is a union of connected (hence irreducible) components of 
	$\wt{X} = \Spa(\wt{R})$. Taking intersections with $\Max(\wt{R})$ reduces 
	the claim to the same statement for the map of Jacobson schemes $p: 
	\Spec(\wt{R})\rightarrow \Spec(R)$, where the claim follows from the fact 
	that $p^{-1}(p(U))-U$ is a union of intersections of distinct irreducible 
	components.

\medskip
To show the second part, we consider $\wt{U}\sub p^{-1}(Z)\sub \wt{X}$. By the first part and Lemma \ref{no4}, $p^{-1}(Z)=\wt{X}$, and hence $Z=X$.
\end{proof}

When $X$ is quasicompact, the normalization relates to the irreducible components in the following way.

\begin{proposition}\label{irr1}
Let $X$ be a quasicompact pseudorigid space over $\ok$. Then the irreducible components of $X$ are exactly the images of the connected components of $\wt{X}$. 
\end{proposition}

\begin{proof}
Let $\wt{X}_{1},...,\wt{X}_{r}$ denote the connected components of $\wt{X}$ and let $X_{i}=p(\wt{X}_{i})$. By Corollary \ref{coh2} the $X_{i}$ are Zariski closed, and they cover $X$ since $p$ is surjective. If $Z_{1},Z_{2}\sub X$ are Zariski closed subsets such that $X_{i}\sub Z_{1}\cup Z_{2}$, then, considering the map $\wt{X}_{i} \ra X$, we may take their preimages to get Zariski closed subsets ${Z}^{\prime}_{1},Z^{\prime}_{2}\sub \wt{X}_{i}$ with $\wt{X}_{i}=Z_{1}^{\prime}\cup Z_{2}^{\prime}$. It follows from Lemma \ref{no4} that $\wt{X}_{i}=Z_{j}^{\prime}$ for some $j\in \{1,2\}$, and hence that $X_{i}\sub Z_{j}$, so $X_{i}$ is irreducible. It remains to show that $X_{i}\sub X_{k}$ only if $i=k$. If $U\sub X$ is open quasicompact, then $\wt{X}_{i}\cap \wt{U}$ and $\wt{X}_{k}\cap \wt{U}$ are closed and open subsets of $\wt{U}$, so one sees that it suffices to verify this statement locally on $X$. We may therefore reduce to the case of an affinoid pseudorigid space, where it follows from the analogous statement for spectra of maximal ideals.
\end{proof}

For non-quasicompact spaces we do not have a decomposition into irreducible components a priori, but we will now show that Proposition \ref{irr1} holds in this case as well. Recall that we have defined irreducible components as maximal irreducible Zariski closed sets (with respect to inclusion).

\begin{proposition}\label{irr2}
Let $X$ be a pseudorigid space over $\ok$ with normalization $p:\wt{X} \ra X$. 
Let $(\wt{X}_{i})_{i\in I}$ be the set of connected components of $\wt{X}$ 
(here $I$ is some index set) and let $X_{i}=p(\wt{X}_{i})$. Then the $X_{i}$ 
are exactly the irreducible components of $X$, and their union is $X$.
\end{proposition}

\begin{proof}
That the $X_{i}$ are irreducible, distinct and cover $X$ follows as in the 
proof of Proposition \ref{irr1}; it remains to show that they are maximal 
irreducible sets. Let $Z\sub X$ be a nonempty irreducible Zariski closed set. 
Let $U\sub X$ be a quasicompact open set with intersects $Z$. By 
quasicompactness of $p$ we can find a finite set $S\subset I$ such that 
$\wt{U}\sub \bigcup_{i\in S}\wt{X}_{i}$, and hence $U \sub \bigcup_{i\in 
S}X_{i}$. Since $U\cap Z\sub Z$ is open and contained in the Zariski closed 
subset $\bigcup_{i\in S}Z\cap X_{i}$ of the irreducible set $Z$, we must have 
$Z\sub X_{i}$ for some $i\in S$, as desired.
\end{proof}

\subsection{Dimension theory}

If $X=\Spa(A,A^{+})$ is an affinoid adic space, then it is natural to define its dimension as the Krull dimension of the spectral space $X$. This definition globalizes in a natural way. However, for our purposes it will be more convenient to use a more algebraic definition for the dimension of a pseudorigid space. Before we define the dimension we make an ad hoc definition which is somewhat overdue.

\begin{definition}
Let $X$ be a pseudorigid space over $\ok$ and let $x\in \Max(X)$. Then we 
define the \emph{completed local ring $\wh{\oo}_{X,x}$ of $X$ at $x$} to be the ring 
$\wh{A_{\mf{m}}}$, where $U=\Spa(A) \sub X$ is any open affinoid pseudorigid 
space containing $x$ and $\mf{m}$ is the maximal ideal in $A$ corresponding to 
$x$. Note that this is independent of the choice of $U$ by Proposition 
\ref{ps2}(3), hence well defined.
\end{definition}

In this section we will freely make use of the fact that a Noetherian local ring has the same Krull dimension as its completion.

\begin{definition}
Let $X$ be a pseudorigid space over $\ok$. We define the \emph{dimension} of $X$ to be 
$$ \dim X=\sup_{x\in \Max(X)} \dim \wh{\oo}_{X,x} $$
(taken to be $+\infty$ if the supremum does not exist). We say that $X$ is \emph{equidimensional} if $\dim \wh{\oo}_{X,x}$ is independent of $x\in \Max(X)$.
\end{definition}

A few remarks are in order. First, we could also have defined $\dim X$ using open affinoid pseudorigid spaces; one has
$$ \dim X = \sup_{U\sub X} \dim \oo_{X}(U), $$
where $U$ runs through the open affinoid pseudorigid subspaces of $X$. When 
$X=\Spa(R)$ is an affinoid pseudorigid space, we have $\dim X=\dim R$. Second, 
we would expect this definition to agree with the Krull dimension of the 
locally spectral space $X$ (for rigid analytic varieties this was proved by 
Huber, cf. \cite[Lemma 1.8.6, Proposition 1.8.11]{hu3}), but we have not tried 
to prove this and we will not need it\footnote{This has now been proved to be true by Louren\c{c}o, see \cite[Corollary 4.13]{lou}.}.

\medskip

The main result we will need on dimensions is that if $X$ is irreducible, then $X$ is equidimensional.

\begin{lemma}\label{dim1}
Let $\mf{m}$ be a maximal ideal of $A=\ok[[X_{1},...,X_{m}]]\langle T_{1},...,T_{n} \rangle$. Then $\dim \wh{A_{\mf{m}}}= m+n+1$.
\end{lemma}

\begin{proof}
This is probably well known, but we sketch a proof for the convenience of the reader. Let $I\sub A$ be the ideal generated by $\pi_{K},X_{1},...,X_{m}$; this is an ideal of definition of $A$, and $A$ is complete with respect to $I$, so $I$ is in the Jacobson radical of $A$ (in fact it is the Jacobson radical). The maximal ideals of $A$ are therefore in bijection with those of $A/I=k[T_{1},...,T_{n}]$, hence parametrised by elements in $\ol{k}^{n}$ (where $\ol{k}$ is an algebraic closure of $k$). Making a finite unramified extension of $K$ if necessary, we may assume that $\mf{m}$ is defined by a tuple in $k^{n}$, and applying a translation we may assume that this tuple is $0$. In other words, $\mf{m}=\langle \vp,X_{1},...,X_{m},T_{1},...,T_{n} \rangle$. Then $\wh{A_{\mf{m}}}\cong \ok[[X_{1},...,X_{m},T_{1},...,T_{n}]]$, which has dimension $m+n+1$.
\end{proof}

\begin{corollary}\label{dim2}
Let $R$ be a Tate ring topologically of finite type over $\ok$, and let $X=\Spa(R)$. Assume that $R$ is an integral domain. Then $X$ is equidimensional.
\end{corollary}

\begin{proof}
Let $R_{0}\sub R$ be a ring of definition which is formally of finite type over 
$\ok$, and choose a surjection $A=\ok[[X]]\langle T_{1},...,T_{n} \rangle 
\twoheadrightarrow R_{0}$. The kernel of 
this 
surjection is a prime ideal which we will call $P$. Let $\mf{m}$ be a maximal 
ideal in $R$ and put $\mf{p}=\mf{m}\cap R_{0}$. Recall that $R_{0}/\mf{p}$ is 
local of dimension $1$, and let $\mf{q}\sub R_{0}$ denote the unique maximal 
ideal of $R_{0}$ above $\mf{p}$. Let $Q$ denote the preimage of $\mf{q}$ in 
$A$. We have
$$  \dim R_{\mf{m}} = \dim (R_{0})_{\mf{p}}. $$
Since $R_{0}$ is catenary, we have
$$ \dim (R_{0})_{\mf{p}} = \dim (R_{0})_{\mf{q}} - 1. $$
Since $A$ is catenary we have
$$ \dim (R_{0})_{\mf{q}} = \dim A_{Q} - \dim A_{P} = n+2 - \dim A_{P}; $$
where we have used Lemma \ref{dim1} in the last equality. Therefore $\dim R_{\mf{m}} = n+1 - \dim A_{P}$ is independent of $\mf{m}$, as desired.
\end{proof}

We can now globalize this.

\begin{theorem}\label{dim3}
Let $X$ be an irreducible pseudorigid space over $\ok$. Then $X$ is equidimensional, and $\dim X = \dim \wt{X}$.
\end{theorem}

\begin{proof}
First assume that $X$ is normal. Let $U=\Spa(R)\sub X$ be an open connected affinoid pseudorigid space. Then $R$ is a normal domain, and therefore $U$ is equidimensional by Corollary \ref{dim2}. By connectedness of $X$ it follows that $X$ is equidimensional.

\medskip

We now do the general case. By the above we know that $\wt{X}$ is 
equidimensional. Let $x\in \Max(X)$. Choose an open affinoid pseudorigid space 
$U=\Spa(R)\sub X$ containing $x$ and let $\mf{m}\sub R$ be the maximal ideal 
corresponding to $x$. We want to show that $\dim R_{\mf{m}}=\dim \wt{X}$. Let 
$\wt{R}$ be the normalization of $R$. $\wt{U}=\Spa(\wt{R})$ is the preimage of 
$U$ in $\wt{X}$ by construction. By the going-up theorem, we know that
$$ \dim R_{\mf{m}} = \max_{\wt{\mf{m}}} \dim \wt{R}_{\wt{\mf{m}}}; $$
where $\wt{\mf{m}}$ ranges over the maximal ideals of $\wt{R}$ lying over $\mf{m}$. Since $\wt{X}$, and therefore $\wt{U}$, is equidimensional, the right hand side is equal to $\dim \wt{X}$, as desired.
\end{proof}

We end this subsection by recording the relation between the irreducible components passing through a maximal point $x\in X$ and the minimal primes of the completed local ring $\wh{\oo}_{X,x}$.

\begin{proposition}\label{dim4}
Let $X$ be a pseudorigid space over $\ok$ and let $x\in \Max(X)$. There is canonical surjective map $\Psi$ from the minimal primes $\mf{p}$ of $\wh{\oo}_{X,x}$ to the irreducible components of $X$ containing $x$, and $\dim \Psi(\mf{p})=\dim \wh{\oo}_{X,x}/\mf{p}$.
\end{proposition} 

\begin{proof}
If $Y$ is a pseudorigid space, let $Irr(Y)$ denote the set of irreducible 
components of $Y$, and if $y\in Y$ is a maximal point, we let $Irr(Y,y)\sub 
Irr(Y)$ be the set of irreducible components containing $y$. Choose an open 
affinoid pseudorigid $U=\Spa(R)\sub X$ containing $x$. Then the morphism 
$\wt{U}\ra \wt{X}$ induces a map $Irr(U)=Irr(\wt{U})\ra Irr(\wt{X})=Irr(X)$, 
which restricts to a surjection $Irr(U,x) \twoheadrightarrow Irr(X,x)$, and 
preserves dimensions.

\medskip
By our definitions, $Irr(U,x)$ can be identified with the set 
$Irr(\Max(R),\mf{m})$ of irreducible components of $\Max(R)$ containing the 
maximal ideal $\mf{m}\sub R$ corresponding to $x$, and this identification 
preserves dimensions. $Irr(\Max(R),\mf{m})$ can in turn be identified with the 
minimal primes of $R_{\mf{m}}$. Since $R_{\mf{m}}\rightarrow \wh{R_{\mf{m}}}$ 
is faithfully flat, it follows from the going down property and the invariance 
of dimension on completion that there is a 
dimension-preserving surjection from the minimal primes of 
$\wh{\oo}_{X,x}=\wh{R_{\mf{m}}}$ to $Irr(\Max(R),\mf{m})$, given by pullback 
along $R \ra \wh{R_{\mf{m}}}$. Finally, one takes the composition with the map 
above to get a map to $Irr(X,x)$, and check that this is independent of the 
choice of $U$.
\end{proof}

\subsection{Factoriality of normalized weight space}

The theory of Fredholm series and hypersurfaces is fundamental to the theory of 
eigenvarieties, which is our intended application. The study of irreducible 
components of Fredholm hypersurfaces provided motivation for developing the 
general theory of irreducible components of rigid spaces. It is based upon the 
notion of a rigid space $X$ being \emph{locally relatively factorial} (in $m$ 
variables), which means that $X$ has a cover by open affinoids 
$U_{i}=\Sp(A_{i})$ such that the relative Tate algebras $A_{i}\langle 
X_{1},\dots X_{m} \rangle$ are factorial (i.e. are unique factorization 
domains). 
When one moves away from eigenvarieties that are equidimensional over weight 
space, the role of Fredholm hypersurfaces becomes less important. For this 
reason, we have not pursued the generalization of the results of \cite[\S 
4]{con} to the setting of pseudorigid spaces (we believe that this should be 
mostly straightforward, but we have not checked the details). One thing, 
however, that is perhaps not so clear, is that the weight spaces that occur for 
the extended eigenvarieties in \cite{jn} are relatively factorial. While we 
will not need this fact in this paper, it seems worth recording. Our proof is 
based upon the following criterion for factoriality. Recall that an integral 
domain $R$ is \emph{locally factorial} if the localizations $R_{\mf{p}}$ are 
factorial for all prime ideals $\mf{p}$ of $R$ (it suffices to check this for 
maximal ideals).

\begin{lemma}\label{lem:samuelufd}
	Let $R$ be a locally factorial Noetherian ring. Suppose $R$ is an integral 
	domain and that there is an element $x$ of the Jacobson radical of $R$ such 
	that $R/xR$ is factorial. Then $R$ is factorial.
\end{lemma}
\begin{proof}
This is \cite[Ch. 2, Lemma 2.2]{samuelufd}; we sketch the proof for the convenience of the reader. It suffices to 
show that every height one prime ideal in $R$ is principal. Let $I \subset R$ 
be a height one prime ideal. By the locally factorial assumption, $I$ is 
locally principal and is therefore projective as an $R$-module. Note that since 
$R/xR$ is a domain, $xR$ is a prime ideal.

\medskip
Suppose $x \in I$. Then $xR \subset I$ and since $I$ has height one we have $xR 
= I$ and therefore $I$ is principal. Now suppose $x \notin I$. Then $I \cap xR 
= xI$, so $I/xI = I/I\cap xR \cong (I+xR)/xR$.  Since $I$ is projective over 
$R$, 
$(I+xR)/xR$ is a 
projective ideal in $R/xR$, and by factoriality of $R/xR$ it is therefore a 
principal ideal. By Nakayama's lemma, $I$ is principal.
\end{proof}

The weight spaces that occur in \cite{jn} have the form $\Spa(\Zp[[S]])^{an}$, 
where $S\cong F \times \Zp^n$ as $p$-adic analytic groups for some finite group 
$F$. It might happen that $F$ has $p$-torsion, in which case 
$\Spa(\Zp[[S]])^{an}$ may fail to be a domain locally. One can take the 
normalization, in which case it will be locally of the form 
$\Zp[[T_{1},\dots,T_{n}]]^{an}$. We will sketch a proof that these spaces are 
locally relatively factorial. Let $K$ be a discrete valued field.

\begin{definition}
Let $X=\Spa(A)$ be an affinoid pseudorigid space over $\ok$. We say that $X$ is 
\emph{relatively factorial} (in $m$ variables) if $A\langle 
X_{1},\dots,X_{m}\rangle$ is factorial. If $X$ is a general pseudorigid space, 
we say that $X$ is \emph{locally relatively factorial} (in $m$ variables) if 
there is an open cover of $X$ by affinoid pseudorigid spaces 
$U_{i}=\Spa(A_{i})$ such that $A_{i}$ is relatively factorial in $m$ variables.
\end{definition}

Now consider the pseudorigid space $\ok[[T_{1},\dots,T_{n}]]^{an}$. It has an open cover given by the affinoid pseudorigid space $U_{0}=\{|T_{1}|,\dots,|T_{n}|\leq |p|\neq 0\}$ and
$$ U_{i}=\{|p|,|T_{1}|,\dots,|T_{n}|\leq |T_{i}|\neq 0 \}, $$
for $i=1,\dots,n$. $U_{0}$ is the adic spectrum of a Tate algebra in $n$ variables over $K$, and hence relatively factorial (in any number of variables) by \cite[5.2.6/1]{bgr}.

\begin{theorem}
Let $m\in \Z_{\geq 0}$ and let $1\leq i \leq n$. Then $U_{i}$ is relatively factorial in $m$ variables. 
\end{theorem}

\begin{proof}
We will content ourselves with a sketch, since this result will not be used in the rest of the paper. By symmetry, it suffices to do the case $i=1$. Set $U:=U_{1}=\Spa(R)$ and $A=R\langle X_{1},\dots,X_{m} \rangle$. $A$ has a ring of definition $A_{0}=R^{\circ}\langle X_{1},\dots,X_{m} \rangle$ which carries the $T:=T_{1}$-adic topology; these are Noetherian rings. We may define an increasing filtration on $A$ by $F_{j}=T^{-j}A_{0}$, and one may compute the corresponding graded ring $gr(A)$; we have
$$gr(A) = k[u_{0},u=u_{1},u^{-1},u_{2},\dots,u_{n},v_{1},\dots,v_{n}]$$
where $k$ is the residue field of $K$, $u_{0}$ is the symbol of $p$ and has degree $1$, $u_{i}$ is the symbol of $T_{i}$ and has degree $1$, and $v_{j}$ is the symbol of $X_{j}$ and has degree $0$. In particular, $gr(A)$ is a domain, so $A$ and $A_{0}$ are domains. We have 
$$ gr(A_{0})= k[u,u_{0}u^{-1},u_{2}u^{-1},\dots,u_{n}u^{-1},v_{1},\dots,v_{m}] $$
so
$$ A_{0}/TA_{0}=k[u_{0}u^{-1},u_{2}u^{-1},\dots,u_{n}u^{-1},v_{1},\dots,v_{m}].$$
We now wish to apply Lemma \ref{lem:samuelufd} to show that $A_{0}$, and hence 
$A$, is factorial. We have verified that $A_{0}$ is a Noetherian domain and 
that $A_{0}/TA_{0}$ is factorial, and that $T$ is in the Jacobson radical of 
$A_{0}$ (since $A_{0}$ is $T$-adically complete). It remains to show that 
$A_{0}$ is locally factorial. Let $\mf{m}$ be a maximal ideal of $A_{0}$; it 
contains $T$. By above the ring $A_{0}/TA_{0}$ is regular, so 
$(A_{0}/TA_{0})_{\mf{m}}=(A_{0})_{\mf{m}}/T(A_{0})_{\mf{m}}$ is regular. Since 
$T$ is a non-zerodivisor, it follows that $(A_{0})_{\mf{m}}$ is a regular local 
ring \cite[\href{http://stacks.math.columbia.edu/tag/00NU}{Tag 
00NU}]{stacks-project} and hence factorial. 
\end{proof}

\section{An interpolation theorem}\label{sec interpolation}

In \cite{che2}, Chenevier introduced an abstract interpolation theorem that allows one to show that a set-theoretic map between subsets of two eigenvarieties, in certain circumstances, extends to a rigid analytic morphism of \emph{reduced} eigenvarieties. Chenevier's interpolation theorem was formulated in terms of the input datum of the general eigenvariety construction of Buzzard \cite{bu}. In \cite{han}, Hansen proved a generalization of Chenevier's theorem for the abstract eigenvariety construction considered in that paper, and gave some applications. In this section we prove a generalization of Hansen's theorem that we believe is close to optimal, and also applies to the extended eigenvarieties constructed in \cite{jn}.

\subsection{Eigenvariety data}

We abstract the ingredients of the construction of extended eigenvarieties in \cite[\S 4]{jn}, generalizing \cite[Definition 4.2.1]{han}. We fix a complete discretely valued field $K$ with ring of integers $\ok$ and a uniformizer $\pi_{K}$.

\begin{definition}
An \emph{eigenvariety datum} is a tuple 
$\mf{O}=(\ms{Z},\ms{H},\mb{T},\psi)$ consisting of a pseudorigid 
space $\ms{Z}$ over $\ok$, $\ms{H}$ a coherent $\oo_{\ms{Z}}$-module, 
$\mb{T}$ a commutative $\Zp$-algebra, and $\psi: \mb{T} \ra 
\End_{\oo_{\ms{Z}}}(\ms{H})$ a $\Zp$-algebra homomorphism.
\end{definition}

We remark that in practice, an eigenvariety datum as above is the penultimate 
step in the construction of an eigenvariety. In these situations, $\ms{Z}$ is 
typically (isomorphic to) $\A^{n}_{\ms{W}}$ for some $n\geq 1$ (or $\mb{G}_{m,\ms{W}}^{n}$), where $\ms{W}$ is the `weight 
space'. In fact, in most eigenvariety constructions one will have $n=1$ and 
$\ms{Z}$ can alternatively be taken to be a Fredholm hypersurface in 
$\A^{1}_{\ms{W}}$; this is where the use of the letter `Z' comes from. Let us 
recall that a Fredholm series over $\ms{W}$ is an entire power series 
$F(X)\in \oo_{\A^{1}_{\ms{W}}}(\A^{1}_{\ms{W}})$ with $F(0)=1$, and a Fredholm 
hypersurface is a Zariski closed subspace of $\A^{1}_{\ms{W}}$ defined by 
$\{F=0\}$, where $F$ is a Fredholm series. 

\medskip
As in \cite[Theorem 4.2.2]{han}, an 
eigenvariety datum $\mf{O}$ has an associated eigenvariety. Before giving the construction, we record the following commutative algebra result (cf. \cite[Lemma 
2.2(1)]{tay}).

\begin{lemma}\label{nilpotence}
Let $A$, $B$ and $T$ be rings with $A$ and $B$ Noetherian and let $M$ be a finitely generated $A$-module. Assume that we have ring homomorphisms $f : A \ra B$ and $\psi_{A} : T \ra \End_{A}(M)$ and let $\psi_{B}$ be the composition of $\psi_{A}$ with the natural map $\End_{A}(M) \ra \End_{B}(M\otimes_{A}B)$ coming from $f$. Let $T_{A}$ (resp. $T_{B}$) be the $A$-subalgebra (resp. $B$-subalgebra) generated by the image of $\psi_{A}$ (resp. $\psi_{B}$). Then the natural $B$-linear map $T_{A}\otimes_{A}B \ra T_{B}$ is a surjection with nilpotent kernel in general, and if $f$ is flat then it is an isomorphism.
\end{lemma}

\begin{proof}
We have $B$-linear maps
$$ T\otimes_{\Z}B \ra \End_{A}(M)\otimes_{A}B \ra \End_{B}(M\otimes_{A}B) $$
induced by $\psi_{A}$ and $f$ respectively. The left map factors through $T_{A}\otimes_{A}B$ and $T_{B}$ is the image of the composition of the two maps. This gives the natural map and shows its surjectivity since $T\otimes_{\Z}B \ra T_{A}\otimes_{A}B$ is surjective. To prove that the kernel is nilpotent, it suffices to show that the support of the $T_{A}\otimes_{A}B$-module $M\otimes_{A}B$ is all of $\Spec(T_{A}\otimes_{A}B)$, since $M\otimes_{A}B$ is a faithful $T_{B}$-module by construction. Since $M\otimes_{A}B \cong M\otimes_{T_{A}}(T_{A}\otimes_{A}B)$, the support of $M\otimes_{A}B$ is the preimage of the support of $M$ under the natural map $\Spec(T_{A}\otimes_{A}B) \ra \Spec(T_{A})$. But $M$ is a faithful $T_{A}$-module by construction, so we get what we want.

\medskip
Now assume that $f$ is flat. Then the natural map $\End_{A}(M)\otimes_{A}B \ra \End_{B}(M\otimes_{A}B)$ is an isomorphism and $T_{A}\otimes_{A}B$ is the image of $T\otimes_{\Z}B \ra \End_{A}(M)\otimes_{A}B$, which gives that $T_{A}\otimes_{A}B \ra T_{B}$ is an isomorphism.
\end{proof}

Let us now record the construction of the eigenvariety of an eigenvariety datum.

\begin{proposition}\label{eigenvariety}
Given an eigenvariety datum $\mf{O}=(\ms{Z},\ms{H},\mb{T},\psi)$, there is a pseudorigid space 
$\ms{X}=\ms{X}(\mf{O})$ over $\ok$, together with a finite morphism $\pi: 
\ms{X} \ra \ms{Z}$, a $\Zp$-algebra homomorphism $\phi_{\ms{X}}: \mb{T} \ra 
\oo(\ms{X})$, and a faithful coherent $\oo_{\ms{X}}$-module $\ms{H}^{\dagger}$. 
There is 
a canonical isomorphism $\pi_{\ast}\ms{H}^{\dagger}\cong \ms{H}$, which is 
compatible with the actions of $\mb{T}$ on both sides (via $\phi_{\ms{X}}$ and 
$\psi$, respectively).

The space $\ms{X}$ is characterized by the following local description: for $U 
\subset \ms{Z}$ a pseudorigid affinoid open we have $\ms{X}_U = \Spa(\ms{T}_U)$ 
where $\ms{T}_{U}$ is the $\oo_{\ms{Z}}(U)$-subalgebra of 
$\End_{\oo_{\ms{Z}}(U)}(\ms{H}(U))$ generated by the image of $\mb{T}$. Note that 
$\ms{H}(U)$ is canonically a $\ms{T}_U$-module, and this gives 
$\ms{H}^{\dagger}$ over $\ms{X}_U$.
\end{proposition}

\begin{proof}
The proof is (essentially) the same as the proof of \cite[Theorem 4.2.2]{han}, 
we sketch it for completeness. For $U\sub \ms{Z}$ 
affinoid open pseudorigid, $\ms{T}_{U}$ is commutative and finite over 
$\oo_{\ms{Z}}(U)$, and hence is a Tate ring formally of finite type over $\ok$. 
The space $\ms{X}_{U}=\Spa(\ms{T}_{U})$ carries a canonical finite map 
$\ms{X}_{U} \ra U$ and $\ms{H}(U)$ is a finitely generated 
$\ms{T}_{U}$-module. By Lemma \ref{nilpotence} and flatness of rational 
localization for affinoid pseudorigid spaces over $\ok$, these 
constructions glue together and satisfy the assertions of the theorem.   
\end{proof}

From Lemma \ref{nilpotence} we get the following compatibility of the eigenvariety construction with base change.

\begin{proposition}\label{basechange}
Let $\mf{O}=(\ms{Z},\ms{H},\mb{T},\psi)$ be an eigenvariety datum with 
eigenvariety $\ms{X}$. Let $f : \ms{Z}^{\prime} \ra \ms{Z}$ be a map of 
pseudorigid spaces over $\ok$. Form the eigenvariety datum 
$$\mf{O}^{\prime}=(\ms{Z}^{\prime}, \ms{H}^{\prime}=f^{\ast}\ms{H}, \mb{T}, \psi^{\prime}), $$
where $\psi^{\prime}$ is the composition of $\psi$ with the natural map $\End_{\oo_{\ms{Z}}}(\ms{H}) \ra \End_{\oo_{\ms{Z}^{\prime}}}(\ms{H}^{\prime})$. Let $\ms{X}^{\prime}$ be the eigenvariety attached to $\mf{O}^{\prime}$. Then there is a natural map $\ms{X}^{\prime} \ra \ms{X}$ over $\ms{Z}$, and the induced map $\ms{X}^{\prime} \ra \ms{X}\times_{\ms{Z}}\ms{Z}^{\prime}$ induces an isomorphism $(\ms{X}^{\prime})^{red} \ra (\ms{X}\times_{\ms{Z}}\ms{Z}^{\prime})^{red}$. If $f$ is flat, then $\ms{X}^{\prime} \ra \ms{X}\times_{\ms{Z}}\ms{Z}^{\prime}$ is an isomorphism.
\end{proposition}

\begin{proof}
By the local nature of the construction of eigenvarieties in Proposition \ref{eigenvariety}, the assertion is local both on $\ms{Z}^{\prime}$ and $\ms{Z}$, so we may assume that they are both affinoid pseudorigid spaces over $\ok$. Then the proposition follows directly from Lemma \ref{nilpotence}.
\end{proof}

We single out of a special case of Proposition \ref{basechange} which 
characterises the points of eigenvarieties.

\begin{corollary}\label{points}
Let $\mf{O}=(\ms{Z},\ms{H},\mb{T},\psi)$ be an eigenvariety datum with eigenvariety $\pi : \ms{X} \ra \ms{Z}$. Fix $z\in \Max(\ms{Z})$. Then the set $\pi^{-1}(z)\sub \Max(\ms{X})$ is in natural bijection with the systems of Hecke eigenvalues appearing in the fibre $\ms{H}(z)$, i.e. the maximal ideals lying above the kernel of the natural map $\mb{T}\otimes_{\Zp}k(z) \ra \End_{k(z)}(\ms{H}(z))$.
\end{corollary}

\begin{proof}
This follows from applying Proposition \ref{basechange} with $f$ the closed immersion $\ms{Z}^{\prime}=z \hookrightarrow \ms{Z}$.
\end{proof}

We finish with a simple reconstruction theorem for eigenvariety data.

\begin{proposition}\label{reconstruction}
Let $\pi : \ms{X} \ra \ms{Z}$ be a finite morphism of pseudorigid spaces and let $\mb{T}$ be a $\Zp$-algebra. Assume that there is a ring homomorphism $\phi : \mb{T} \ra \oo(\ms{X})$ such that $\oo_{\ms{X}}$ is generated by the image of $\phi$ over $\oo_{\ms{Z}}$, and assume that $\ms{H}^{\dagger}$ is a faithful coherent $\oo_{\ms{X}}$-module. We may form an eigenvariety datum
$$ \mf{O}=(\ms{Z},\ms{H},\mb{T},\psi), $$
where $\ms{H}=\pi_{\ast}\ms{H}^{\dagger}$ and $\psi$ is the composition $\mb{T} \overset{\phi}{\ra} \oo(\ms{X}) \ra \End_{\oo_{\ms{Z}}}(\ms{H})$. Then $\pi : \ms{X} \ra \ms{Z}$ is the eigenvariety attached to $\mf{O}$, and $\phi=\phi_{\ms{X}}$.
\end{proposition}

\begin{proof}
The assertion is local on $\ms{Z}$, so we may assume that $\ms{Z}=\Spa(A)$, 
$\ms{X}=\Spa(T)$, that $\ms{H}^{\dagger}$ is the coherent sheaf attached to a 
finitely generated faithful $T$-module $H$, and that the map 
$\mb{T}\otimes_{\Zp}A \ra T$ induced by $\phi$ is surjective. The eigenvariety 
attached to $\mf{O}$ is then the adic spectrum of the image of the map 
$\mb{T}\otimes_{\Zp}A \ra \End_{A}(H)$. This map factors as
$$ \mb{T} \otimes_{\Zp}A \ra T \ra \End_{T}(H) \ra \End_{A}(H) $$
and the first map is surjective and the second and third are injective (using that $H$ is a faithful $T$-module), so this image is (canonically isomorphic to) $T$ and the identifications of $\pi$ and $\phi_{\ms{X}}$ follow easily.
\end{proof}

We remark that, as a result of this, any Zariski closed subset $\ms{X}^{\prime}$ of the eigenvariety $\ms{X}$ of an eigenvariety datum  is naturally the eigenvariety of an eigenvariety datum.

\subsection{The interpolation theorem}

We are now ready to prove our interpolation theorem (which could also be considered as a rigidity theorem). If $\ms{U}$ is a pseudorigid space over $\ok$ and $\ms{M}$ is a coherent $\oo_{\ms{U}}$-module, then we continue to write $\ms{M}(u)$ for the fibre of $\ms{M}$ at any $u\in \Max(\ms{U})$, which is a vector space over the residue field $k(u)$.

\begin{theorem}\label{optimalint}
Let $\mf{O}_{i}=(\ms{Z}_{i},\ms{H}_{i},\mb{T}_{i},\psi_{i})$, for $i=1,2$, be eigenvariety data. Assume that we have the following data:
\begin{itemize}
\item A morphism $j : \ms{Z}_{1} \ra \ms{Z}_{2}$ of adic spaces;

\item A $\Zp$-algebra homomorphism $\sigma : \mb{T}_{2} \ra \mb{T}_{1}$;

\item A subset $\ms{X}^{cl}\sub \Max(\ms{X}_{1})$ such that the 
$\mb{T}_{2}$-eigensystem of $x$ (i.e. the $\mb{T}_{1}$-eigensystem of $x$ 
composed with $\sigma$) appears in $\ms{H}_{2}(j(\pi_{1}(x)))$ for all $x\in 
\ms{X}^{cl}$.
\end{itemize}
Let $\ol{\ms{X}}$ denote the Zariski closure of $\ms{X}^{cl}$ in $\ms{X}_{1}$, with 
its induced reduced structure. Then there is a canonical morphism 
$$ i : \ol{\ms{X}} \ra \ms{X}_{2} $$
lying over $j : \ms{Z}_{1} \ra \ms{Z}_{2}$ such that $\phi_{\ol{\ms{X}}}\circ \sigma = i^{\ast}\circ \phi_{\ms{X}_{2}}$. The morphism $i$ inherits the following properties from $j$:

\begin{itemize}
\item If $j$ is (partially) proper (resp. finite), then $i$ is (partially) proper (resp. finite);
 
\item If $j$ is a closed immersion and $\sigma$ is a surjection, then $i$ is a closed immersion.
\end{itemize}
\end{theorem}

\begin{proof}
We start with a series of reduction steps. First, we form the eigenvariety datum 
$$ \mf{O}_{1}^{\sigma}=(\ms{Z}_{1},\ms{H}_{1},\mb{T}_{2},\psi_{1}\circ \sigma) $$
with corresponding eigenvariety $\pi_{1}^{\sigma} : \ms{X}_{1}^{\sigma} \ra \ms{Z}_{1}$. If $U\sub \ms{Z}_{1}$ is an open affinoid, we have a natural inclusion $\oo_{\ms{X}_{1}^{\sigma}}((\pi_{1}^{\sigma})^{-1}(U)) \sub \oo_{\ms{X}_{1}}(\pi_{1}^{-1}(U)) $ directly from the definitions, and they glue together to a finite dominant (in fact surjective) map
$ f : \ms{X}_{1} \ra \ms{X}_{1}^{\sigma}$ over $\ms{Z}_{1}$. Note that if $\sigma$ is a surjection, then the inclusion in the previous sentence is an equality, and $f$ is an equality. Form the analogue $\ol{\ms{X}}^{\sigma}$ of $\ol{\ms{X}}$ for $\ms{X}_{1}^{\sigma}$. One checks easily that $f(\ol{\ms{X}})\sub \ol{\ms{X}}^{\sigma}$, so it suffices to prove the theorem after replacing $\mf{O}_{1}$ by $\mf{O}_{1}^{\sigma}$. In other words, we may assume that $\mb{T}_{1}=\mb{T}_{2}=:\mb{T}$.
\medskip

Next, we pull back the eigenvariety datum $\mf{O}_{2}$ along $j :\ms{Z}_{1} \ra \ms{Z}_{2}$ to reduce to the case $\ms{Z}_{1}=\ms{Z}_{2}=:\ms{Z}$, $j=id$, by Proposition \ref{basechange}. To check that this reduction works, note that if $x\in \ms{X}^{cl}$, then the $\mb{T}$-eigensystem of $x$ appears in $\ms{H}_{2}(j(\pi_{1}(x))$ if and only if it appears in $(j^{\ast}\ms{H}_{2})(\pi_{1}(x))$, since $(j^{\ast}\ms{H}_{2})(\pi_{1}(x)) = \ms{H}_{2}(j(\pi_{1}(x))\otimes_{k(j(\pi_{1}(x)))}k(\pi_{1}(x))$.

\medskip
Having done these reductions, we form the eigenvariety datum
$$ \mf{O}_{3}=(\ms{Z},\ms{H}_{3}=\ms{H}_{1}\oplus \ms{H}_{2},\mb{T},\psi_{3}=(\psi_{1},\psi_{2})). $$
Let $\pi_{3}: \ms{X}_{3} \ra \ms{Z}$ be the associated eigenvariety; it has both $\ms{X}_{1}$ and $\ms{X}_{2}$ appearing as Zariski closed subspaces in it since the coherent $\oo_{\ms{Z}}$-algebra $\pi_{3,\ast}\oo_{\ms{X}_{3}}$ has both $\pi_{1,\ast}\oo_{\ms{X}_{1}}$ and $\pi_{2,\ast}\oo_{\ms{X}_{2}}$ naturally as quotients (this follows by examining the construction in Proposition \ref{eigenvariety}). We need to show that $\ms{X}^{cl}\sub \ms{X}_{2}$; if so then $\ol{\ms{X}}\sub \ms{X}_{2}$ since $\ms{X}_{2}$ is Zariski closed in $\ms{X}_{3}$. Let $x\in \ms{X}^{cl}$ and set $z=\pi_{3}(x)=\pi_{1}(x)$. By Corollary \ref{points}, we need to show that the $\mb{T}$-eigensystem of $x$ appears in $\ms{H}_{2}(z)$. But this is exactly our assumption. This establishes the existence of $i$. For uniqueness, we can reduce to the case when $\ms{Z}_{1}$ and $\ms{Z}_{2}$ are both affinoid (since equality of morphisms can be checked locally on the source), and then the requirement $\phi_{\ol{\ms{X}}}\circ \sigma = i^{\ast}\circ \phi_{\ms{X}_{2}}$ uniquely determines $i^{\ast}$, and hence $i$, since the image of $\phi_{\ol{\ms{X}}}$ generates $\oo(\ol{\ms{X}})$ over $\oo(\ms{Z}_{1})$.

\medskip
It remains to prove that $i$ inherits properties from $j$. If we summarize the construction, $i$ is a composition of a finite map $\ol{\ms{X}} \ra \ol{\ms{X}}^{\sigma}$, a closed immersion $\ol{\ms{X}}^{\sigma} \ra (\ms{X}_{2}\times_{\ms{Z}_{2}}\ms{Z}_{1})^{red}$ and the canonical morphism $(\ms{X}_{2}\times_{\ms{Z}_{2}}\ms{Z}_{1})^{red} \ra \ms{X}_{2}$ coming from the projection. If $j$ is (partially) proper (resp. finite), then $(\ms{X}_{2}\times_{\ms{Z}_{2}}\ms{Z}_{1})^{red} \ra \ms{X}_{2}$ is (partially) proper (resp. finite), and hence the composition is (partially) proper (resp. finite, since these properties are stable under base change and composition), proving the first assertion. If $\sigma$ is surjective, then $\ol{\ms{X}} \ra \ol{\ms{X}}^{\sigma}$ is an isomorphism (as noted above), so if $j$ is a closed immersion, then $(\ms{X}_{2}\times_{\ms{Z}_{2}}\ms{Z}_{1})^{red} \ra \ms{X}_{2}$ is a closed immersion and hence $i$ is a closed immersion.
\end{proof}

\begin{remark}\label{rem:compareinterp}
We should note that this theorem is not, strictly speaking, a generalization of 
\cite[Theorem 5.1.6]{han}. That theorem assumes a certain divisiblity of 
determinants instead of the assumption on eigensystems in our theorem. This 
divisibility is a weaker assumption; from such a result one typically deduces a 
result about eigensystems by a separation of eigenvalues argument. In practice 
(e.g. attempts to interpolate known cases of Langlands functoriality) this 
separation of eigenvalues has been done, so we think that our slightly stronger 
assumption is natural.

With this caveat, our theorem appears to be essentially optimal, and the (rather elementary) method of proof appears to be new. We note in particular that the use of the global geometry of the eigenvariety instead of a reduction to affinoids eliminates the need for $\ms{X}^{cl}$ to be \emph{very Zariski dense}. It should be noted, however, that the only technique currently known (to the authors) to control the Zariski closure of interesting sets $\ms{X}^{cl}$ occurring in practice is to show that they are very Zariski dense inside a union of irreducible components of $\ms{X}$. 
\end{remark}

\subsection{Extended eigenvarieties for overconvergent 
cohomology}\label{sec:evars}
In this section we briefly recall the extended eigenvarieties constructed in 
\cite{jn}. We refer to \cite[\S 3.3, \S 4]{jn} for precise definitions and any 
undefined notation. Let $F$ be a number field and let $\mbf{H}/F$ be a 
connected reductive group which is split at all places above $p$. We set 
$\mbf{G}={\rm Res}_{\Q}^{F}\mbf{H}$. Choosing split models 
$\mbf{H}_{\oo_{F_{v}}}$ of $\mbf{H}$ over $\oo_{F_{v}}$ for all $v|p$ and  
maximal tori and Borel subgroups $\mbf{T}_{v}\sub \mbf{B}_{v}\sub 
\mbf{H}_{\oo_{F_{v}}}$, we obtain a model $\mbf{G}_{\Zp}=\prod_{v|p}{\rm 
Res}_{\Zp}^{\oo_{F_{v}}}\mbf{H}_{\oo_{F_{v}}}$ of $\mbf{G}$ over $\Zp$, and 
closed subgroup schemes $\mbf{T}=\prod_{v|p}{\rm 
Res}_{\Zp}^{\oo_{F_{v}}}\mbf{T}_{v}\sub \mbf{B}=\prod_{v|p}{\rm 
Res}_{\Zp}^{\oo_{F_{v}}}\mbf{B}_{v}\mbf{G}_{\Zp}$. Set $T_{0}=\mbf{T}(\Zp)$ and 
let $I$ be the preimage of $\mbf{B}(\Fp)$ under the map $\mbf{G}_{\Zp}(\Zp) \ra 
\mbf{G}_{\Zp}(\Fp)$. We define $\Sigma$ to be the kernel of a choice of 
splitting of the inclusion $T_{0}\sub \mbf{T}(\Qp)$. Inside $\Sigma$, we have a 
certain submonoid $\Sigma^{+}$ and a subset $\Sigma^{cpt}\sub \Sigma^{+}$. Fix 
compact open subgroups $K_{\ell}\sub \mbf{G}(\Q_{\ell})$ for all $\ell\neq p$ 
such that $K_{\ell}=\mc{G}(\Z_{\ell})$ for all but finitely many $\ell$, where 
$\mc{G}$ is a reductive model of $\mbf{G}$ over $\Z[1/M]$ for some $M\in 
\Z_{\geq 1}$. Set $K^{p}=\prod_{\ell\neq p}K_{\ell}$ and $K=K^{p}I$. Let 
$\mbf{Z}$ denote the center of $\mbf{G}$, put $Z(K)=\mbf{Z}(\Q)\cap K$ and let 
$\ol{Z(K)}\sub T_{0}$ be the $p$-adic closure. Finally let $K_{\infty}\sub 
\mbf{G}(\R)$ be a maximal compact and connected subgroup and let 
$Z_{\infty}^\circ\sub Z_\infty = \mbf{Z}(\R)$ be the identity component.

\medskip
If $R$ is a Banach--Tate $\Zp$-algebra \cite[\S 3.1]{jn} and $\kappa : 
T_{0}/\ol{Z(K)} \ra R^{\times}$ is a continuous character, then (under the 
assumption that the norm of $R$ is adapted to $\kappa$ \cite[Def. 3.3.2]{jn}) 
we defined distribution modules $\mc{D}_{\kappa}^{r}$ for $r<1$ that are close 
enough to $1$. $\mc{D}_{\kappa}^{r}$ has actions of $I$ and $\Sigma^{+}$, and 
elements of $\Sigma^{cpt}$ act as compact operators. $\mc{D}_{\kappa}^{r}$ may 
be considered as a local system on the locally symmetric space 
$X_{K}=\mbf{G}(\Q)\setminus 
G(\A)/KK_{\infty}Z_{\infty}^\circ$. A choice of triangulation of the 
Borel--Serre 
compactification of $X_{K}$ and choices of homotopies between the corresponding 
simplicial chain complex and the singular chain complex gives a complex 
$C^{\bu}(K,\mc{D}_{\kappa}^{r})$ that computes the cohomology 
$H^{\ast}(X_{K},\mc{D}_{\kappa}^{r})$. Let $\mb{T}_{\ell}$ be the spherical 
Hecke algebra with respect to $K_{\ell}$ for any $\ell\neq p$ such that 
$K_{\ell}=\mc{G}(\Z_{\ell})$ and set $\mb{T}=\bigotimes_{\ell}\mb{T}_{\ell}$. 
Choosing an element $t\in \Sigma^{cpt}$, one may define an eigenvariety datum
$$ (\ms{Z}, \ms{H}, \mb{T}, \psi) $$
when $R$ is a Tate ring formally of finite type over $\Zp$, using the method of 
\cite[\S 4]{jn}. Roughly speaking, the coherent sheaf $\ms{H}$ is constructed 
out of the finite slope part of $H^{\ast}(X_{K},\mc{D}_{\kappa}^{r})$ with 
respect to the Hecke operator $U_{t}=[KtK]$. $\ms{Z} \subset \mathbb{A}^1_R$ is 
the 
Fredholm hypersurface for the Hecke operator $U_{t}$ acting on  the complex
$C^{\bu}(K,\mc{D}_{\kappa}^{r})$ (this action depends on the choices of 
homotopies we made above, and in particular will only commute up to homotopy 
with $U_{t'}$ for a different choice of $t' \in \Sigma^{cpt}$). The 
homomorphism $\psi$ comes from the action 
of $\mb{T}$ 
on $H^{\ast}(X_{K},\mc{D}_{\kappa}^{r})$. The eigenvariety datum is independent 
of the choice of $r$ and the choice of norm on $R$, and is compatible with open 
immersions $\Spa(S) \ra \Spa(R)$. If we replace $\ms{Z}$ by $\mb{A}^1_R$, and 
$\ms{H}$ by its pushforward under the closed immersion $\ms{Z} \subset 
\mb{A}^1_R$ we obtain an eigenvariety datum (with the same associated 
eigenvariety) which is moreover independent of 
our chosen triangulation of the Borel--Serre compactification of $X_{K}$ and 
the choices of homotopies between the corresponding simplicial chain complex 
and the singular chain complex. However, the associated eigenvariety will in 
general depend on the choice of controlling operator $U_t$, since we have not 
incorporated any other Hecke operators at $p$ into the eigenvariety datum. We 
refer to section \ref{independence} for more discussion of this issue.

\medskip
Fix $R$ and $\ka$ as above and put $\ms{W}=\Spa(R)$. Let $w\in \Max(\ms{W})$ 
with residue field $k(w)$ and write $\ka_{w}$ for the induced character 
$T_{0}/\ol{Z(K)} \ra k(w)^{\times}$. Consider the eigenvariety datum 
$(\ms{Z}, \ms{H}, \mb{T}, \psi)$ and let $z=(w,\lambda)\in 
\Max(\ms{Z})$. 
The following proposition is a simple corollary 
of 
\cite[Corollary 4.2.3]{jn}, and will be used in the next section to interpolate 
some known cases of Langlands functoriality.

\begin{proposition}\label{prop:evals}
The systems of eigenvalues for $\mb{T}$ occurring in the fibre $\ms{H}(z)$ are 
the same as the systems of eigenvalues of $\mb{T}$ occurring in the generalized 
$\lambda^{-1}$-eigenspace of $U_{t}$ on $H^{\ast}(K,\mc{D}^{r}_{w})$ (for any 
allowed choice of $r$).
\end{proposition}

\begin{proof}
Choosing a slope datum $(U,h)$ for $\ms{Z}$ 
with $z\in \ms{Z}_{U,h}$ (see \cite[Def.~2.3.1]{jn} for the 
notion of a slope datum for a Fredholm hypersurface), then by construction 
$\ms{H}(\ms{Z}_{U,h})=H^{\ast}(K, \mc{D}_{U}^{r})_{\leq h}$ (and this is 
independent of $r$). Set
$$ \mb{T}_{w,h}={\rm Im}\left( \mb{T}\otimes_{\Zp}k(w) \ra \End_{k(w)}(H^{\ast}(K, \mc{D}_{w}^{r})_{\leq h}) \right); $$
$$ \mb{T}_{U,h}={\rm Im}\left( \mb{T}\otimes_{\Zp}\oo(U) \ra \End_{\oo(U)}(H^{\ast}(K, \mc{D}_{U}^{r})_{\leq h}) \right). $$
By \cite[Corollary 4.2.3]{jn} we have 
$(\mb{T}_{U,h}\otimes_{\oo(U)}k(w))^{red}\cong \mb{T}_{w,h}^{red}$, so the 
systems of eigenvalues for $\mb{T}$ occurring in the spaces 
$H^{\ast}(K,\D_{w}^{r})_{\leq h}$ and $H^{\ast}(K,\D_{U}^{r})_{\leq 
h}\otimes_{\oo(U)}k(w)$ agree. Let $\mf{m}\sub \oo(\ms{Z}_{U,h})$ be the 
maximal ideal corresponding to $z$. Now $ \ms{H}(z) = 
H^{\ast}(K,\D_{U}^{r})_{\leq h}\otimes_{\oo(\ms{Z}_{U,h})}k(z)$ and the system 
of eigenvalues of $\mb{T}$ occurring in $\ms{H}(z)$ are the maximal ideals of 
$\mb{T}_{U,h}\otimes_{\oo(Z_{U,h})}k(z)$, by Lemma \ref{nilpotence}. By the 
same 
lemma, these are the same as the systems of eigenvalues occurring in 
$\mb{T}_{U,h}\otimes_{\oo(Z_{U,h})}\oo(\ms{Z}_{U,h})/\mf{m}^{n}$, for any 
$n\geq 1$. By the above, these systems of eigenvalues are the same as those 
occurring in $H^{\ast}(K,\D_{w}^{r})_{\leq 
h}\otimes_{\oo(\ms{Z}_{U,h})}\oo(\ms{Z}_{U,h})/\mf{m}^{n}$, and for $n \gg 1$ 
this is the generalized $\lambda^{-1}$-eigenspace of $U_{t}$.
\end{proof}

To finish this subsection, we note that the second author's result 
\cite[Proposition B.1]{han} giving a lower bound for the dimensions of 
irreducible components of eigenvarieties (see also \cite[Prop.~5.7.4]{urb}) 
holds for 
extended eigenvarieties. We retain the notation from above, and write $\pi : 
\ms{X} \ra \ms{Z}$ for the eigenvariety attached to 
$(\ms{Z},\ms{H},\mb{T},\psi)$.

\begin{proposition}\label{prop:dimension}
Let $x\in \Max(\ms{X})$ and put $(w,\lambda)=\pi(x)$. Let $h \in \Q_{\ge 0}$ be 
such that $H^{*}(K,\D_{w}^{r})_{\le h, x}\neq 0$.  Set
$$ l(x)= \sup \left\{ i \mid H^{i}(K,\D_{w}^{r})_{\le h, x}\neq 0\right\} - 
\inf \left\{ i \mid H^{i}(K,\D_{w}^{r})_{\le h, x}\neq 0\right\}.$$ Note that 
$l(x)$ is independent of the choice of $h$. Assume that $\wh{\oo}_{\ms{W},w}$ 
is Cohen--Macaulay and, for simplicity, that $\ms{W}$ is equidimensional. Then 
the dimension of any irreducible component containing $x$ is greater than or 
equal to $\dim \ms{W} -l(x)$. 
\end{proposition}
\begin{proof}
Choose a slope datum $(U,h)$ such that $\pi(x)\in \ms{Z}_{U,h}$ and set 
$\ms{X}_{U,h}=\pi^{-1}(\ms{Z}_{U,h})$. By Proposition \ref{dim4}, it suffices 
to show that any minimal prime of $\wh{\oo}_{\ms{X},x}$ has coheight $\geq \dim 
\ms{W} -l(x)$. From the construction of the eigenvariety datum, we get a 
complex $C^{\bu}(K,\D_{U}^{r})_{\leq h}$ of finite projective $\oo(U)$-modules 
that computes $\ms{H}(\ms{Z}_{U,h})=H^{\ast}(K,\D_{U}^{r})_{\leq h}$. Applying 
$-\otimes_{\oo(\ms{X}_{U,h})}\wh{\oo}_{\ms{X},x}$, we get a complex 
$C^{\bu}_{x}$ of finite projective $\wh{\oo}_{\ms{W},w}$-modules that computes 
$H_{x}=\ms{H}(\ms{Z}_{U,h})\otimes_{\oo(\ms{X}_{U,h})}\wh{\oo}_{\ms{X},x}$. 
Since $H_{x}$ is a finite faithful $\wh{\oo}_{\ms{X},x}$-module, it suffices to 
show that any minimal prime in the support of $H_{x}$ in $\wh{\oo}_{\ms{W},w}$ 
has height $\leq l(x)$. This follows from \cite[Theorem 2.1.1(1)]{han2}, upon 
noting that $l(x)$ is equal to the amplitude of $C^{\bu}_x$, in the notation of 
\emph{loc.~cit}.
\end{proof}

\subsection{Independence of the choice of controlling 
operator}\label{independence}
The eigenvariety construction of the previous subsection involve making  a 
choice of `controlling operator' $U_t$, depending on a choice of $t \in 
\Sigma^{cpt}$. In this subsection, we describe a variant construction which 
incorporates all Atkin--Lehner Hecke operators at $p$ into the eigenvariety 
construction. We then show that this construction gives an 
eigenvariety which 
is independent of the choice of $t \in \Sigma^{cpt}$.

We retain all the notations of the previous subsection\footnote{In particular, 
our weight space $\ms{W}$ is affinoid, but everything discussed in this section 
glues to handle the general situation.}, and begin by defining 
some additional notation. We have a commutative subalgebra of the 
Iwahori--Hecke algebra at $p$ \[\ms{A}_p^+ \subset \Zp[\mathbf{G}(\Qp)//I] \] 
generated by the characteristic functions $\mathbf{1}_{[IsI]}$ for all $s \in 
\Sigma^+$ (see \cite[Prop.~6.4.1]{bc}). 

Now choosing an element $t\in \Sigma^{cpt}$, one may define an eigenvariety 
datum
$$ (\ms{Z}^t, \ms{H}^t, \mb{T}\otimes_{\Zp}\ms{A}_p^+, \psi^t) $$ and 
associated 
eigenvariety $\ms{X}^{\ms{A}_p^+,t}$ with $\phi^t: 
\mb{T}\otimes_{\Zp}\ms{A}_p^+ 
\rightarrow \oo(\ms{X}^{\ms{A}_p^+,t})$ as in the previous section (we now 
explicitly record the dependence on $t$ in the notation), by 
incorporating 
the action of $\ms{A}_p^{+}$ on $H^{\ast}(X_{K},\mc{D}_{\kappa}^{r})$.

\begin{lemma}\label{invertible}
For all $s \in \Sigma^+$ the action of $[IsI]$ on $\ms{H}^t$ is invertible. 
Equivalently, $\phi^t([IsI])$ is a unit in $\oo(\ms{X}^{\ms{A}_p^+,t})$.
\end{lemma}
\begin{proof}
	This follows from the two following claims: the action of $U_t = [ItI]$ on 
	$\ms{H}^t$ 
	is invertible and for any $s \in \Sigma^+$ there exists $k \ge 0$ and $s' 
	\in \Sigma^+$
	such that $ss' = t^k$. The first claim is an immediate consequence of the 
	construction of $\ms{H}^t$ using slope decompositions for $U_t$. For the 
	second claim we use the notation in 
	\cite[\S 3.3]{jn}: there is an 
	integer $r \ge 1$ such that $s^{-1}\overline{N}_r s \subset 
	\overline{N}_1$, and for $k$ sufficiently large we have $t^k\overline{N}_1 
	t^{-k} \subset \overline{N}_r$. This implies that $s^{-1}t^k \in 
	\Sigma^+$.\end{proof}

Now we denote by $\widehat{\Sigma}=\Hom(\Sigma,\mathbb{G}_{m,\ms{W}})$ the 
pseudorigid space 
over $\ms{W}$ representing $S \mapsto \Hom(\Sigma, S^\times)$ for affinoid 
pseudorigid $R$-algebras $S$. If we fix a $\Z$-basis for $\Sigma$ we get an 
isomorphism $\widehat{\Sigma} \cong 
\mathbb{G}_{m,\ms{W}}^{\rk_{\Z}(\Sigma)}$. 

For each $s \in \Sigma$ we write $s 
= 
s'(s'')^{-1}$ for $s',s'' \in \Sigma^+$ and obtain an element 
\[\phi([Is'I])\phi([Is''I])^{-1} \in  \oo(\ms{X}^{\ms{A}_p^+,t})^\times.\]

This defines a map $\pi^{\ms{A}_p^+, t}: \ms{X}^{\ms{A}_p^+,t} \rightarrow 
\widehat{\Sigma}$ which is finite because the finite map 
$\pi: 
\ms{X}^{\ms{A}_p^+,t} \rightarrow 
\ms{Z} \subset \mb{A}^1_{\ms{W}}$ is a composition of $\pi^{\ms{A}_p^+, t}$ and 
the separated map $\widehat{\Sigma} \rightarrow 
\mb{A}^1_{\ms{W}}$ given by evaluation at $t^{-1}$. Applying Proposition 
\ref{reconstruction}, we see that $\ms{X}^{\ms{A}_p^+,t}$ can be identified 
with the eigenvariety 
associated to the eigenvariety datum \[(\widehat{\Sigma}, 
\pi^{\ms{A}_p^+, 
t}_*\ms{H}^{t,\dagger},\mb{T}\otimes_{\Zp}\ms{A}_p^+,\psi^{\ms{A}_p^+, t} 
)\] where $\psi^{\ms{A}_p^+, t}$ is the composition of $\phi^t$ with the map 
$\oo(\ms{X}^{\ms{A}_p^+,t}) \ra \End_{\oo_{\widehat{\Sigma}}}(\pi^{\ms{A}_p^+, 
	t}_*\ms{H}^{t,\dagger})$.

We can now show that the nilreduction of $\ms{X}^{\ms{A}_p^+,t}$ is independent 
of the choice of $t \in \Sigma^{cpt}$. In particular, $U_{t'}^{-1}$ induces a 
finite map $(\ms{X}^{\ms{A}_p^+,t})^{red} \rightarrow \mb{A}^1_{\ms{W}}$ and we 
will use this consequence to show that $\ms{X}^{\ms{A}_p^+,t}$ itself is 
independent of the choice of $t$.

\begin{lemma}\label{nilredindependent}
Let $t, t' \in \Sigma^{cpt}$. There is a canonical 
$\widehat{\Sigma}$-isomorphism 
$$(\ms{X}^{\ms{A}_p^+,t})^{red} \cong (\ms{X}^{\ms{A}_p^+,t'})^{red}$$ which is 
compatible with the maps 
$\phi^t:\mb{T}\otimes_{\Zp}\ms{A}_p^+ 
\rightarrow \oo(\ms{X}^{\ms{A}_p^+,t})^{red}$ and 
$\phi^{t'}:\mb{T}\otimes_{\Zp}\ms{A}_p^+ 
\rightarrow \oo(\ms{X}^{\ms{A}_p^+,t'})^{red}$.
\end{lemma}
\begin{proof}
	We apply Theorem \ref{optimalint} to the two eigenvariety data 
	\begin{align*}(\widehat{\Sigma}, 
	\pi^{\ms{A}_p^+, 
		t}_*\ms{H}^{t,\dagger},\mb{T}\otimes_{\Zp}\ms{A}_p^+,\psi^{\ms{A}_p^+, 
		t} 
	),~~ (\widehat{\Sigma}, 
	\pi^{\ms{A}_p^+, 
		t'}_*\ms{H}^{t',\dagger},\mb{T}\otimes_{\Zp}\ms{A}_p^+,\psi^{\ms{A}_p^+,
		 t'}),\end{align*} with $j$ and $\sigma$ (in the notation of Theorem 
		 \ref{optimalint}) the identity maps and the subset $\ms{X}^{cl}$ 
		 of `classical points' equal to all the maximal points. A simple 
		 variant of Proposition \ref{prop:evals} gives a canonical bijection 
		 between the maximal points of the two eigenvarieties being considered. 
\end{proof}

Finally, we show that the eigenvariety  $\ms{X}^{\ms{A}_p^+,t}$ (not just the 
nilreduction) is independent of the choice of $t \in \Sigma^{cpt}$. We will 
need a basic lemma about finite maps of pseudorigid spaces:
\begin{lemma}\label{finitered}
	Let $f: \ms{X} \rightarrow \ms{Y}$ be a map of pseudorigid spaces. Suppose 
	that the induced map $f^{red}: \ms{X}^{red} \rightarrow \ms{Y}$ is finite. 
	Then $f$ is finite.
\end{lemma}
\begin{proof}
	Maps of pseudorigid spaces are necessarily locally of finite type (by 
	Proposition \ref{ps2}(\ref{ps2finite}) and Proposition \ref{ps5}). So we 
	may apply \cite[Prop.~1.5.5]{hu3} which says that (for locally of finite 
	type maps of analytic adic spaces) finite is equivalent to quasi-finite and 
	proper. Since $f$ and $f^{red}$ are the same on underlying topological 
	spaces, it is easy to see that the quasi-finiteness and properness of 
	$f^{red}$ implies these properties for $f$. 
\end{proof}

\begin{proposition}\label{prop: independence}
Let $t, t' \in \Sigma^{cpt}$. There is a canonical 
$\widehat{\Sigma}$-isomorphism 
$$\ms{X}^{\ms{A}_p^+,t} \cong \ms{X}^{\ms{A}_p^+,t'}$$ which is 
compatible with the maps 
$\phi^t:\mb{T}\otimes_{\Zp}\ms{A}_p^+ 
\rightarrow \oo(\ms{X}^{\ms{A}_p^+,t})$ and 
$\phi^{t'}:\mb{T}\otimes_{\Zp}\ms{A}_p^+ 
\rightarrow \oo(\ms{X}^{\ms{A}_p^+,t'})$.
\end{proposition}
\begin{proof}
First we note that Lemma \ref{invertible} implies that there is a map 
$\ms{X}^{\ms{A}_p^+,t} \overset{U_{t'}^{-1}}\rightarrow \mb{A}^1_{\ms{W}}$. By 
Lemma \ref{nilredindependent}, this induces a finite map 
$(\ms{X}^{\ms{A}_p^+,t})^{red} \rightarrow \mb{A}^1_{\ms{W}}$ which factors 
through the spectral variety $\ms{Z}^{t'}$.  It follows from Lemma 
\ref{finitered} that the map $U_{t'}^{-1}$ is finite and factors through a 
closed adic subspace $\widetilde{\ms{Z}}^{t'}$ of $\mb{A}^1_{\ms{W}}$ with the 
same 
underlying topological space as $\ms{Z}^{t'}$.

Now we consider the finite map $p : \ms{X}^{\ms{A}_p^+,t} \rightarrow 
\mb{A}_\ms{W}^2$ 
given by $(U_t^{-1},U_{t'}^{-1})$, which factors through the fibre product 
$\ms{Z}:=\ms{Z}^t\times_{\ms{W}}\widetilde{\ms{Z}}^{t'}$. 

By Proposition \ref{reconstruction}, it now suffices to show that the coherent 
sheaf $p_*\ms{H}^{t,\dagger}$ has a description which is symmetric in $t$ and 
$t'$. We can cover $\ms{Z}^t$ by affinoid opens $\ms{Z}^t_{U,h}$, with $(U,h)$ 
running over slope data (for the characteristic power series of $U_t$), and 
similarly for $\ms{Z}^{t'}$. Since $\ms{Z}^{t'} \rightarrow 
\widetilde{\ms{Z}}^{t'}$ is a nilpotent thickening this gives a cover of 
$\widetilde{\ms{Z}}^{t'}$ by affinoid opens $\widetilde{\ms{Z}}^{t'}_{V,k}$. We 
have $\widetilde{\ms{Z}}^{t'}_{V,k} \cong \Spa(\oo(V)[X]/I_{V,k})$ with 
$I_{V,k}$ an ideal of $\oo(V)[X]$ with $(Q'(X)^N) \subset I_{V,k} \subset 
(Q'(X))$ where $Q'(X)$ is the multiplicative polynomial in the slope 
factorisation of the 
characteristic power series of $U_{t'}$ corresponding to the slope datum 
$(V,k)$.

We obtain an open cover of $\ms{Z}$ by the fibre products
$\ms{Z}^t_{U,h}\times_{\ms{W}}\widetilde{\ms{Z}}^{t'}_{V,k}$ which are finite 
over $U\cap V \subset \ms{W}$. Since $\ms{W}$ is quasi-separated, we can cover 
$\ms{Z}$ by affinoid opens  $\ms{Z}_{U,h,k} = 
\ms{Z}^t_{U,h}\times_{\ms{W}}\widetilde{\ms{Z}}^{t'}_{U,k}$, finite over 
affinoid opens $U$ in $\ms{W}$ such that $(U,h)$ is a slope datum for $U_t$ and 
$(U,k)$ is a slope datum for $U_{t'}$. 

We claim that the restriction of $p_*\ms{H}^{t,\dagger}$ to 
$\ms{Z}_{U,h,k}$ is the coherent sheaf associated to 
\[H^{\ast}(X_{K},\mc{D}_{U}^{r})_{U_{t}\le h} \cap 
H^{\ast}(X_{K},\mc{D}_{U}^{r})_{U_{t'}\le k}\] where the subscripts denote 
taking the slope decompositions for $U_t$ and $U_{t'}$, which exist since 
$(U,h)$ and $(U,k)$ are slope data for the respective characteristic power 
series. In particular, we can take $\widetilde{\ms{Z}}^{t'} = \ms{Z}^{t'}$ and 
the symmetry between $t$ and $t'$ in our description of $p_*\ms{H}^{t,\dagger}$ 
implies the Proposition. So it remains to justify our claim.

By definition $\ms{H}^{t}|_{\ms{Z}^{t}_{U,h}}$ is the coherent sheaf associated 
to 
$H^{\ast}(X_{K},\mc{D}_{U}^{r})_{U_{t}\le h}$. Recall that we have 
multiplicative polynomials $Q$ and $Q'$ over $\oo(U)$ such that 
$H^{\ast}(X_{K},\mc{D}_{U}^{r})_{U_{t}\le h} = \ker(Q^*(U_t))$ 
and $Q^*(U_t)$ is an isomorphism on  $H^{\ast}(X_{K},\mc{D}_{U}^{r})_{U_{t}> 
h}$, and similarly for 
the slope $\le k$ decomposition with respect to $U_{t'}$. Since $U_t$ and 
$U_{t'}$ commute (hence $Q^*(U_t)$ and $Q^{\prime *}(U_{t'})$ also commute) on 
$H^{\ast}(X_{K},\mc{D}_{U}^{r})$ the idempotent projectors 
\[(Q^*(U_t)|_{\mathrm{Im} 
Q^*(U_t)})^{-1} \circ Q^*(U_t),~~ (Q^{\prime *}(U_{t'})|_{\mathrm{Im} 
Q^{\prime *}(U_{t'})})^{-1} \circ Q^{\prime *}(U_{t'})\]
giving their slope decompositions commute and we have a decomposition
\[H^{\ast}(X_{K},\mc{D}_{U}^{r})_{U_{t}\le h} = 
H^{\ast}(X_{K},\mc{D}_{U}^{r})_{U_{t}\le h}\cap 
H^{\ast}(X_{K},\mc{D}_{U}^{r})_{U_{t'}\le k} \oplus 
H^{\ast}(X_{K},\mc{D}_{U}^{r})_{U_{t}\le h}\cap 
H^{\ast}(X_{K},\mc{D}_{U}^{r})_{U_{t'}> k}.\]

The restriction of $p_*\ms{H}^{t,\dagger}$ to 
$\ms{Z}_{U,h,k}$ corresponds to the module 
\[H^{\ast}(X_{K},\mc{D}_{U}^{r})_{U_{t}\le 
h}\otimes_{\oo(U)[X]}\oo(U)[X]/I_{U,k}\] where $X$ acts by $U_{t'}^{-1}$. Since 
$Q(U_{t'}^{-1})$ is an isomorphism on 
$H^{\ast}(X_{K},\mc{D}_{U}^{r})_{U_{t'}> k}$ 
and $Q(U_{t'}^{-1})$ is zero on $H^{\ast}(X_{K},\mc{D}_{U}^{r})_{U_{t'}\le k}$, 
this tensor product can be identified with 
$H^{\ast}(X_{K},\mc{D}_{U}^{r})_{U_{t}\le 
h}\cap H^{\ast}(X_{K},\mc{D}_{U}^{r})_{U_{t'}\le k}$, as claimed.
\end{proof}

\section{Application}
We give a sample application of our interpolation theorem 
Thm.~\ref{optimalint}, interpolating cyclic base change and using 
Prop.~\ref{prop:dimension} to show that there exist large characteristic $p$ 
loci in the extended eigenvarieties for $\GL_2$ over number fields.

\subsection{Cyclic base change}

The following theorem is a consequence of \cite{arthurclozel}, see 
\cite[Lem.~5.1.1]{blgg}.
\begin{theorem}\label{ACBC}
	Suppose $L/K$ is a cyclic extension of number fields. Let $\pi$ be a 
	cuspidal automorphic representation of $\GL_n(\A_K)$, and suppose that $\pi 
	\ncong \pi \otimes (\chi\circ\mathrm{Art}_K\circ\det)$ for any character 
	$\chi$ of $\Gal(L/K)$. Then there is a cuspidal automorphic representation 
	$\Pi$ of $\GL_n(\A_L)$ such that for all places $w$ of $L$ lying over a 
	place $v$ of $K$ we have $\mathrm{rec}(\Pi_w) \cong 
	\mathrm{rec}(\pi_v)|_{W_{L_w}}$.
\end{theorem}

\subsection{Examples of eigenvariety data}
\subsubsection{The extended Coleman--Mazur eigencurve}
With the notation of Section \ref{sec:evars} , we consider the case 
$\mbf{G}=\GL_{2/\Q}$. We fix an integer $N \ge 5$ which is prime to $p$. We let 
$\ms{W}$ denote the 
analytic adic weight space parametrising continuous characters $\kappa = 
(\kappa_1,\kappa_2)$ of \[T_0 = \begin{pmatrix}
\Zp^\times & 0 \\ 0 & \Zp^\times
\end{pmatrix}\] (see \cite[Defn.~4.1.1]{jn}) and let $\ms{W}_0 \subset \ms{W}$ 
denote the closed subspace 
where $\kappa_2 = 1$.

We let $K^p \subset \GL_2(\A_f^p)$ be the compact open subgroup given by \[ 
K^p = \{g \in \GL_2(\widehat{\Z}^p) \mid g \equiv \begin{pmatrix}
* & *\\ 0 & 1
\end{pmatrix}\mod N \}\] let $I\subset \GL_2(\Q_p)$ be the upper triangular 
Iwahori subgroup and 
set $K = K^pI \subset \GL_2(\A_f)$.

Let $S$ be a finite set of primes, containing all the primes dividing $pN$, and 
let 
\[\T^S_{\GL_{2/\Q}} = 
\otimes_{l \notin S}\T(\GL_2(\Q_l),\GL_2(\Z_l))\] be the product of spherical 
Hecke 
algebras over $\Zp$ for places away from $S$.

Let $f \in \Z_{\ge 1}$. Using the construction recalled in Section 
\ref{sec:evars}, and gluing over a pseudorigid affinoid open cover of 
$\ms{W}_0$, we 
obtain an eigenvariety datum 
\[(\ms{Z}^f_{\GL_{2/\Q}},\ms{H}_{\GL_{2/\Q}},\T^S_{\GL_{2/\Q}},\psi_{\GL_{2/\Q}})\]
 and associated eigenvariety $\ms{E}^{S,f}$ whose maximal points correspond to 
 systems 
 of 
Hecke eigenvalues for $\T^S$ and $U_p^f$, where $U_p = [K\begin{pmatrix}
1 & 0\\0 & p
\end{pmatrix}K]$, with non-zero $U_p^f$-eigenvalue, appearing in the 
overconvergent cohomology spaces $H^1(K,\mc{D}^r_\kappa)$. 

\begin{remark}
	For any $f \ge 1$ there is a natural finite map $\ms{E}^{S,1} \rightarrow 
	\ms{E}^{S,f}$.
\end{remark}

We now define some notation: if $\ms{X}$ is a pseudorigid space over $\oo_K$ we 
denote by $\ms{X}^{rig} \subset \ms{X}$ the Zariski open subspace of $\ms{X}$ 
given by the 
locus where $\pi_K \ne 0$. We have $\ms{X}^{rig} = \ms{X} 
\times_{\Spa(\oo_K,\oo_K)}\Spa(K,\oo_K)$ and $\ms{X}^{rig}$ is (the adic space 
associated to) a rigid space over $K$.

\begin{lemma}
	The subspace $(\ms{E}^{S,f})^{rig}$ is 
	Zariski dense in $\ms{E}^{S,f}$. In particular, 
	any Zariski 
	dense subset of $(\ms{E}^{S,f})^{rig}$ is Zariski dense in $\ms{E}^{S,f}$.
\end{lemma}
\begin{proof}
	Let $(U,h)$ be a slope datum for $\ms{E}^{S,f}$. We recall (see the 
	discussion after \cite[Lem.~6.1.3]{jn}) that 
	$\ms{E}^{S,f}$ has an open cover by affinoid pseudorigid subspaces 
	$\ms{E}^{S,f}_{U,h}$ 
	where 
	$(U,h)$ is a slope datum, and  $\ms{E}^{S,f}_{U,h}$ is finite 
	over a 
	Fredholm 
	hypersurface $\ms{Z}_{\GL_{2/\Q},U,h}$ for a Fredholm polynomial with 
	coefficients in	$\OO({U})$. Now ${U}^{rig}$ is Zariski 
	dense 
	in ${U}$, so it follows from \cite[Lem.~6.2.8]{chegln} (irreducible 
	components of $\ms{E}^{S,f}_{U,h}$ have dimension $1$, so they surject 
	onto 
	irreducible components of ${U}$ and this lemma applies) that its inverse 
	image 
	$(\ms{E}^{S,f}_{{U},h})^{rig}$ is Zariski dense in 
	$\ms{E}^{S,f}_{{U},h}$. Since 
	the $\ms{E}^{S,f}_{{U},h}$ cover $\ms{E}^{S,f}$ it follows that 
	$(\ms{E}^{S,f})^{rig}$ is 
	Zariski dense in $\ms{E}^{S,f}$.
\end{proof}

We now denote by $\ms{E}^{S,f}_{cusp}$ the Zariski closure of the classical 
cuspidal points of $(\ms{E}^{S,f})^{rig}$, with its induced reduced structure. 
The 
preceding lemma implies that 
$\ms{E}^{S,f}_{cusp}$ is equal to the Zariski closure of 
$(\ms{E}^{S,f})^{rig}_{cusp}$, the Zariski closure in $(\ms{E}^{S,f})^{rig}$ of 
the classical cuspidal 
points, which is a union of irreducible components of 
$(\ms{E}^{S,f})^{rig}$.  This furthermore implies that $\ms{E}^{S,f}_{cusp}$  
is a union of irreducible components of $\ms{E}^{S,f}$. Finally, if $F$ is a 
number field we denote by $\ms{E}^{S,f}_{cusp,F\mhyphen ncm}$ the 
Zariski closure of the classical 
cuspidal points of $(\ms{E}^{S,f})^{rig}$ which do not have CM by an imaginary 
quadratic subfield of $F$. This is again a union of irreducible components of 
$\ms{E}^{S,f}$ and contains all the non-ordinary components.
\subsubsection{Eigenvarieties for $\GL_{2/F}$}
We now allow $F$ to be an arbitrary number field and, with the notation of 
Section \ref{sec:evars}, we consider the case 
$\mbf{H}=\GL_{2/F}$. We let $K^p_F \subset 
\GL_2(\A_{F,f}^p)$ be the compact open subgroup given by 
\[ 
K^p_F = \{g \in \GL_2(\widehat{\OO}_F^p) \mid g \equiv \begin{pmatrix}
* & *\\ 0 & 1
\end{pmatrix}\mod N \}\] let $I_F \subset \prod_{v|p}\GL_2(F_v)$ be the upper 
triangular Iwahori subgroup and 
set $K_F = K^p_FI_F \subset \GL_2(\A_{F,f})$. We let $\ms{W}_F$ denote the 
analytic adic weight space parametrising continuous characters $\kappa$ of 
\[T_{F,0} = 
\begin{pmatrix}
\prod_{v|p}\OO_{F,v}^\times & 0 \\ 0 & \prod_{v|p}\OO_{F,v}^\times
\end{pmatrix}\] which are trivial on (the closure of) the image of $Z(K) = 
F^\times\cap K$ in $T_{F,0}$. 

The dimension of $\ms{W}_F$ is equal to the 
difference between $2[F:\Q]$ and the $\Z_p$-rank of the closure of 
$\OO_F^\times$ in $\OO^\times_{F,p}$. This latter rank is $r_1 + r_2 - 1 - 
\mathfrak{d}_{F,p}$, where $\mathfrak{d}_{F,p}$, by definition, measures the 
defect in Leopoldt's conjecture for $F$ and $p$ and $r_1, r_2$ denote the 
number of real and complex places of $F$.

Again we have an eigenvariety 
datum 
\[(\ms{Z}_{\GL_{2/F}},\ms{H}_{\GL_{2/F}},\T^S_{\GL_{2/F}},\psi_{\GL_{2/F}})\]
where \[\T^S_{\GL_2/F} = \otimes_{v \notin 
	S}\T(\GL_2(F_v),\GL_2(\OO_{F_v}))\] is a tensor product over finite places 
	$v$ 
which do not divide any of the primes in $S$. We denote the associated 
eigenvariety by $\ms{X}^{S}_F$ whose points correspond to systems 
of 
Hecke eigenvalues for $\T^S_{\GL_2/F}$ and $U_{p,F}=\prod_{v|p}U_v$, where $U_v 
= 
[K_F\begin{pmatrix}
1 & 0\\0 & \varpi_v
\end{pmatrix}K_F]$, with non-zero $U_{p,F}$-eigenvalue, appearing in the 
overconvergent cohomology spaces $H^*(K_F,\mc{D}^r_\kappa)$, defined using the 
locally symmetric 
spaces $X_{K_F}$ appearing in \ref{sec:evars}. Note that the dimension of 
$X_{K_F}$ (as a real manifold) is $2r_1 + 3r_2$.

\subsection{Base change}
For $F_v/\Q_l$ we have a map of spherical Hecke algebras 
$\T(\GL_2(F_v),\GL_2(\OO_{F_v})) \rightarrow 
\T(\GL_2(\Q_l),\GL_2(\Z_l))$ 
induced by unramified local Langlands and the map $\rho \mapsto 
\rho|_{W_{F_v}}$ on the Galois side. We make this map explicit.

We write $\T(\GL_2(\Q_l),\GL_2(\Z_l)) 
= \Zp[T^1_l,T^2_l]$, where $T^1_l$ and $T^2_l$ correspond to the double cosets 
of $\begin{pmatrix}
1 & 0 \\ 0 & l
\end{pmatrix}$ and $\begin{pmatrix}
l & 0 \\ 0 & l
\end{pmatrix}$ respectively. If $\lambda: 
\T(\GL_2(\Q_l),\GL_2(\Z_l)) \rightarrow \Qbar_p$ is a homomorphism 
we associate to it the semisimple unramified Weil group representation 
 where $\Frob_l$ (a geometric Frobenius element) has characteristic 
 polynomial $X^2 - 
\lambda(T^1_l)X + l\lambda(T^2_l)$. Likewise, if $\lambda_v: 
\T(\GL_2(F_v),\GL_2(\OO_{F_v})) \rightarrow \Qbar_p$ we associate 
the unramified representation where $\Frob_v$ has characteristic polynomial 
$X^2 - 
\lambda(T^1_v)X + q_v\lambda(T^2_v)$ where $q_v = l^{f_v}$ is the cardinality 
of the residue field of $F_v$ and $l$ is the 
rational prime under $v$. 

It follows that the map $\T(\GL_2(F_v),\GL_2(\OO_{F_v})) \rightarrow 
\T(\GL_2(\Q_l),\GL_2(\Z_l))$ takes $T^1_v$ to $\mathrm{tr}(\Frob_l^{f_v})$ 
(a homogeneous degree $f_v$ polynomial in $T^1_l$ and $lT^2_l$) 
and 
$T^2_v$ to $(T^2_l)^{f_v}$.

We thereby obtain the map of Hecke algebras $\sigma_F^S: \T^S_{\GL_2/F} 
\rightarrow 
\T^S_{\GL_2/\Q}$.

We denote by $j: \ms{W}_0 \rightarrow \ms{W}_F$ the map of weight spaces 
induced by the norm map $T_{F,0} \rightarrow T_0$.

\begin{theorem}\label{padicBC}
	Let $F/\Q$ be a cyclic extension of number fields. Let $g$ be 
	the number of 
	places of $F$ dividing $p$, let $e, f$ be the inertial and residual degrees 
	of a place dividing $p$. There is a canonical 
	finite morphism \[i: 
	\ms{E}^{S,1}_{cusp,F\mhyphen ncm} 
	\rightarrow \cX^S_{F}\] lying over $j: \mathscr{W}_0 \rightarrow 
	\mathscr{W}_F$ and compatible with the map 
	$\sigma_F^S:\T^S_{\GL_2/F} \rightarrow 
	\T^S_{\GL_2/\Q}$.
\end{theorem}
\begin{proof}
	First we apply the finite map $\ms{E}^{S,1}_{cusp,F\mhyphen ncm} 
	\rightarrow \ms{E}^{S,fg}_{cusp,F\mhyphen ncm}$. So it suffices to 
	show that there is a canonical finite morphism \[i: 
	\ms{E}^{S,fg}_{cusp,F\mhyphen ncm} 
	\rightarrow \cX^S_{F}\] with the properties specified by the theorem. 
	The 
	subset $\mathscr{X}^{cl} \subset \ms{E}^{S,fg}$ is 
	defined to 
	be the points 
	arising from classical cusp forms of weight $k\ge 2$, level $K$ and 
	$U_p^{fg}$-slope $< \frac{k-1}{e}$ and which moreover do not have CM by an 
	imaginary quadratic subfield of $F$. Note that  $\mathscr{X}^{cl}$ is 
	Zariski 
	dense 
	in 
	$(\ms{E}^{S,fg}_{cusp,F\mhyphen ncm})^{rig}$, and hence it is Zariski dense 
	in 
	$\ms{E}^{S,fg}_{cusp,F\mhyphen ncm}$. 
	
	By Thm.~\ref{ACBC} (we excluded CM points so the condition of the theorem 
	is satisfied), for each point $x \in 
	\mathscr{X}^{cl}$ we have a 
	cuspidal
	automorphic representation $\pi_x$ of $\GL_2(\A_F)$ which is regular 
	algebraic 
	of 
	weight $(k-2,0)_{\tau: F \rightarrow \C}$ and whose Hecke eigenvalues are 
	giving by pulling back the Hecke eigenvalues for $x$ by the map 
	$\sigma_F^S$. 
	Moreover, for each $v|p$ one of the $U_v$-eigenvalues on the Iwahori 
	invariants 
	of 
	$\pi_{x,v}$ has $p$-adic valuation equal to $f_v$ times the slope of the 
	classical form giving rise to $x$, so there is a $U_{p,F}$-eigenvalue of 
	$\pi_{x,p}^{I_F}$ with $p$-adic valuation $<\frac{k-1}{e} = 
	v_p(\varpi_v^{k-1})$. Now we can apply \cite[Thm.~3.2.5]{han}, together 
	with 
	\cite[Thm.~4.3.3]{han} and Prop.~\ref{prop:evals} to show that the system 
	of Hecke eigenvalues arising from $\pi_x$ appears in 
	$\cX^S_{F}$. Finally, we conclude by applying Thm.~\ref{optimalint} to the 
	eigenvariety data \[\mf{O}_1 = 
	(\mb{A}^1_{\ms{W}_0},\ms{H}_{\GL_{2/\Q}},\T^S_{\GL_{2/\Q}},\psi_{\GL_{2/\Q}}),
	~~\mf{D}_2	 = 
	(\mb{A}^1_{\ms{W}_F},\ms{H}_{\GL_{2/F}},\T^S_{\GL_{2/F}},\psi_{\GL_{2/F}})\]
	 where the map $j: \mb{A}^1_{\ms{W}_0} \rightarrow \mb{A}^1_{\ms{W}_F}$ is 
	induced by $j: \mathscr{W}_0 \rightarrow 
	\mathscr{W}_F$ and we have already defined the map 
	$\sigma_F^S:\T^S_{\GL_2/F} \rightarrow 
	\T^S_{\GL_2/\Q}$.
\end{proof}

\begin{corollary}\label{cor:bigcharp}
	Let $F/\Q$ be a cyclic extension of number fields. Then  
	$(\cX^S_{F})^{rig} \subset \cX^S_{F}$ is a strict inclusion. 
	Moreover, the same is true if we restrict to the non-ordinary locus 
	(i.e.~where 
	the $U_{p,F}$-eigenvalue is not a unit), and the dimension of the Zariski 
	closed subset $\cX^S_{F}\backslash (\cX^S_{F})^{rig}$ (and 
	its non-ordinary locus) is at 
	least 
	$[F:\Q]$. Note that the Leopoldt conjecture is known for $F,p$, so 
	$\ms{W}_F$ 
	has dimension $1+[F:\Q]+r_2$.
\end{corollary}
\begin{proof}
	It follows from \cite[Cor.~A1]{bp} that there is a non-ordinary 
	irreducible component of $\ms{E}^{S,1}_{cusp}$ which has a characteristic 
	$p$ point. Since the component is non-ordinary, the non-CM points are 
	Zariski dense. Applying Theorem \ref{padicBC}, we deduce that there is an 
	irreducible component of $\cX^S_{F}$ which contains a characteristic 
	$p$ point. We also see (from the proof of Theorem \ref{padicBC}), that this 
	component contains a classical point $x$
	arising from a cuspidal
	automorphic 
	representation $\pi$ with a $U_{p,F}$-eigenvalue of 
	$\pi_{p}^{I_F}$ with $p$-adic valuation $< v_p(\varpi_v^{k-1})$. The 
	representation $\pi$ contributes to the cohomology of $X_{K_F}$ precisely 
	in degrees $r_1+r_2,\ldots,r_1+2r_2$ (see \cite[3.6]{Har}) --- moreover 
	this 
	contribution accounts for the entire generalised eigenspace for the 
	associated system of Hecke eigenvalues at places away from $S$, by (for 
	example) the 
	Jacquet--Shalika classification theorem \cite[Thm.~4.4]{JS}. It follows 
	from \cite[Thm.~3.2.5]{han} that we have $l(x)=r_2$ and Proposition 
	\ref{prop:dimension} therefore implies that the irreducible component under 
	consideration has 
	dimension $\ge 1 + [F:\Q]$ and so the 
	characteristic $p$ locus in this irreducible component has dimension $\ge 
	[F:\Q]$.
\end{proof}
\begin{remark}\label{rem: extension}
Using other known cases of Langlands functoriality it is possible to produce 
more examples of $p$-adic functoriality which show the existence of a large 
characteristic $p$ locus in extended eigenvarieties. For example, using 
solvable base change and Dieulefait's results on base change for $\GL_2$ 
\cite{Die,DieII}, one can extend the above corollary to eigenvarieties for 
$\GL_2$ over a solvable 
extension $F'$ of a totally real field $F$. One could also consider the 
symmetric square lifting and show the existence of large characteristic $p$ 
loci in extended eigenvarieties for $\GL_3/\Q$.

\medskip
The reader may also wonder, in view of \S \ref{independence}, whether Corollary 
\ref{cor:bigcharp} remains true if we had included $\ms{A}_{p}^{+}$ in the 
construction of the eigenvariety. This is true; it follows from the fact that, 
if we denote this eigenvariety by $\ms{X}_{F}^{S,\ms{A}_{p}^{+}}$, then by the 
construction of these eigenvarieties there is a canonical finite surjective map 
$\ms{X}_{F}^{S,\ms{A}_{p}^{+}} \ra \ms{X}_{F}^{S}$. Similar remarks apply if 
one adds or removes other Hecke operators (using Proposition \ref{prop: 
independence} if one wishes to change controlling operator). One can also prove 
a version of Theorem \ref{padicBC} using the eigenvarieties incorporating the 
full Atkin--Lehner algebra $\ms{A}_{p}^{+}$ --- the norm map between tori 
induces a map between the Atkin--Lehner Hecke algebras for $F$ and for $\Q$, 
which is easily seen to be compatible with base change functoriality for 
automorphic representations which are unramified principal series or unramified 
twist of Steinberg at $p$.
\end{remark}
\bibliographystyle{alpha}
\bibliography{halo}

\begin{thebibliography}{BLGG11}

\bibitem[Abb10]{ab}
Ahmed Abbes.
\newblock {\em \'{E}l\'ements de g\'eom\'etrie rigide. {V}olume {I}}, volume
  286 of {\em Progress in Mathematics}.
\newblock Birkh\"auser/Springer Basel AG, Basel, 2010.
\newblock Construction et {\'e}tude g{\'e}om{\'e}trique des espaces rigides.
  [Construction and geometric study of rigid spaces], With a preface by Michel
  Raynaud.

\bibitem[AC89]{arthurclozel}
James Arthur and Laurent Clozel.
\newblock {\em Simple algebras, base change, and the advanced theory of the
  trace formula}, volume 120 of {\em Annals of Mathematics Studies}.
\newblock Princeton University Press, Princeton, NJ, 1989.

\bibitem[AIP]{aip}
Fabrizio Andreatta, Adrian Iovita, and Vincent Pilloni.
\newblock Le halo spectral.
\newblock To appear in Ann.~Sci.~ENS, available at
  \url{http://www.mat.unimi.it/users/andreat/SpectralHalo.pdf}.

\bibitem[BC09]{bc}
Jo{\"e}l Bella{\"{\i}}che and Ga{\"e}tan Chenevier.
\newblock Families of {G}alois representations and {S}elmer groups.
\newblock {\em Ast\'erisque}, (324):xii+314, 2009.

\bibitem[BGR84]{bgr}
Siegfried Bosch, Ulrich G{\"u}ntzer, and Reinhold Remmert.
\newblock {\em Non-{A}rchimedean analysis}, volume 261 of {\em Grundlehren der
  Mathematischen Wissenschaften [Fundamental Principles of Mathematical
  Sciences]}.
\newblock Springer-Verlag, Berlin, 1984.
\newblock A systematic approach to rigid analytic geometry.

\bibitem[BLGG11]{blgg}
Thomas Barnet-Lamb, Toby Gee, and David Geraghty.
\newblock The {S}ato-{T}ate conjecture for {H}ilbert modular forms.
\newblock {\em J. Amer. Math. Soc.}, 24(2):411--469, 2011.

\bibitem[BP]{bp}
John Bergdall and Robert Pollack.
\newblock Arithmetic properties of {F}redholm series for $p$-adic modular
  forms.
\newblock Preprint, \url{http://arxiv.org/abs/1506.05307}.

\bibitem[Buz07]{bu}
Kevin Buzzard.
\newblock Eigenvarieties.
\newblock In {\em {$L$}-functions and {G}alois representations}, volume 320 of
  {\em London Math. Soc. Lecture Note Ser.}, pages 59--120. Cambridge Univ.
  Press, Cambridge, 2007.

\bibitem[Che04]{chegln}
Ga{\"e}tan Chenevier.
\newblock Familles {$p$}-adiques de formes automorphes pour {${\rm GL}\sb n$}.
\newblock {\em J. Reine Angew. Math.}, 570:143--217, 2004.

\bibitem[Che05]{che2}
Ga{\"e}tan Chenevier.
\newblock Une correspondance de {Jacquet}--{Langlands} {$p$}-adique.
\newblock {\em Duke Math. J.}, 126(1):161--194, 2005.

\bibitem[Con99]{con}
Brian Conrad.
\newblock Irreducible components of rigid spaces.
\newblock {\em Ann. Inst. Fourier (Grenoble)}, 49(2):473--541, 1999.

\bibitem[Die12]{Die}
Luis Dieulefait.
\newblock Langlands base change for {${\rm GL}(2)$}.
\newblock {\em Ann. of Math. (2)}, 176(2):1015--1038, 2012.

\bibitem[Die15]{DieII}
Luis Dieulefait.
\newblock Automorphy of {${\rm Symm}^5({\rm GL}(2))$} and base change.
\newblock {\em Journal de Math\'{e}matiques Pures et Appliqu\'{e}es},
  104(4):619 -- 656, 2015.

\bibitem[Gul]{gu}
Daniel Gulotta.
\newblock Equidimensional adic eigenvarietes for groups with discrete series.
\newblock Available at \url{http://www.math.columbia.edu/~dgulotta/}.

\bibitem[Hana]{han2}
David Hansen.
\newblock Iwasawa theory of overconvergent modular forms, {I}: {C}ritical
  $p$-adic {$L$}-functions.
\newblock Preprint, \url{http://arxiv.org/abs/1508.03982}.

\bibitem[Hanb]{han}
David Hansen.
\newblock Universal eigenvarieties, trianguline {G}alois representations, and
  $p$-adic {L}anglands functoriality.
\newblock To appear in Crelle, \url{http://arxiv.org/abs/1412.1533}.

\bibitem[Har87]{Har}
G.~Harder.
\newblock Eisenstein cohomology of arithmetic groups. {T}he case {${\rm
  GL}_2$}.
\newblock {\em Invent. Math.}, 89(1):37--118, 1987.

\bibitem[Hub93]{hu1}
Roland Huber.
\newblock Continuous valuations.
\newblock {\em Math. Z.}, 212(3):455--477, 1993.

\bibitem[Hub94]{hu2}
Roland Huber.
\newblock A generalization of formal schemes and rigid analytic varieties.
\newblock {\em Math. Z.}, 217(4):513--551, 1994.

\bibitem[Hub96]{hu3}
Roland Huber.
\newblock {\em \'{E}tale cohomology of rigid analytic varieties and adic
  spaces}.
\newblock Aspects of Mathematics, E30. Friedr. Vieweg \& Sohn, Braunschweig,
  1996.

\bibitem[JN16]{jn}
C.~{Johansson} and J.~{Newton}.
\newblock {Extended eigenvarieties for overconvergent cohomology}.
\newblock {\em ArXiv e-prints}, April 2016.
\newblock \url{https://arxiv.org/abs/1604.07739}.

\bibitem[JS81]{JS}
H.~Jacquet and J.~A. Shalika.
\newblock On {E}uler products and the classification of automorphic
  representations. {I}.
\newblock {\em Amer. J. Math.}, 103(3):499--558, 1981.

\bibitem[{Lou}17]{lou}
J.~N.~P. {Louren\c{c}o}.
\newblock {The Riemannian Hebbarkeitss\"atze for pseudorigid spaces}.
\newblock {\em ArXiv e-prints}, 2017.
\newblock \url{https://arxiv.org/abs/1711.06903}.

\bibitem[LWX]{lwx}
Ruochuan Liu, Daqing Wan, and Liang Xiao.
\newblock Eigencurve over the boundary of the weight space.
\newblock To appear in {D}uke {M}ath.~J., \url{http://arxiv.org/abs/1412.2584}.

\bibitem[Sam64]{samuelufd}
P.~Samuel.
\newblock {\em Lectures on unique factorization domains}.
\newblock Notes by M. Pavman Murthy. Tata Institute of Fundamental Research
  Lectures on Mathematics, No. 30. Tata Institute of Fundamental Research,
  Bombay, 1964.

\bibitem[{Sta}16]{stacks-project}
The {Stacks Project Authors}.
\newblock {\itshape Stacks Project}.
\newblock \url{http://stacks.math.columbia.edu}, 2016.

\bibitem[{Tay}08]{tay}
Richard {Taylor}.
\newblock Automorphy of some {$l$}-adic lifts of automorphic mod {$l$} galois
  representations. ii.
\newblock {\em Publ. Math. Inst. Hautes Etudes Sci.}, 108:183--239, 2008.

\bibitem[Urb11]{urb}
Eric Urban.
\newblock Eigenvarieties for reductive groups.
\newblock {\em Ann. of Math. (2)}, 174(3):1685--1784, 2011.

\bibitem[Val75]{v1}
Paolo Valabrega.
\newblock On the excellent property for power series rings over polynomial
  rings.
\newblock {\em J. Math. Kyoto Univ.}, 15(2):387--395, 1975.

\bibitem[Val76]{v2}
Paolo Valabrega.
\newblock A few theorems on completion of excellent rings.
\newblock {\em Nagoya Math. J.}, 61:127--133, 1976.

\end{thebibliography}

\end{document}